\documentclass[12pt,reqno]{amsart}
\AtBeginDocument{\addtocontents{toc}{\protect\setlength{\parskip}{0pt}}}

\usepackage{amsfonts}
\usepackage{dsfont}
\usepackage{amsmath}
\usepackage{amssymb}
\usepackage{mathrsfs}
\usepackage{url}
\usepackage{caption}
\usepackage{enumitem}
\usepackage{graphicx, subfig}
\usepackage[breaklinks]{hyperref}

\usepackage{tikz} 
\usetikzlibrary{decorations.fractals} 
\usetikzlibrary{calc}

\usepackage{dsfont}
\newtheorem{theorem}{Theorem}[section]
\newtheorem{lemma}[theorem]{Lemma}
\newtheorem{proposition}[theorem]{Proposition}

\theoremstyle{definition}
\newtheorem{definition}[theorem]{Definition}
\newtheorem{example}[theorem]{Example}

\theoremstyle{remark}
\newtheorem{remark}[theorem]{Remark}

\numberwithin{equation}{section}

\DeclareMathOperator{\supp}{supp}

\begin{document}

\title{Quantitative FUP and spectral gap for quasi-Fuchsian group}

\author{Long Jin, An Zhang and Hong Zhang}

\address{Yau Mathematical Sciences Center, Jingzhai, Tsinghua University, Haidian District, Beijing, 100084, P.R.China}
\email{jinlong@mail.tsinghua.edu.cn}
\email{yuukizaizh@gmail.com}

\address{School of Mathematical Sciences, Beihang University, 37 Xueyuan road, Haidian district, Beijing, 100191, P.R.China}
\email{anzhang@buaa.edu.cn}

\date{\today}

\begin{abstract}
We derive an explicit formula for the exponent $\beta$  in the higher‑dimensional fractal uncertainty principle (FUP) established by Cohen \cite{MR4927737}, quantifying its dependence on the porosity parameter $\nu$ of the Fourier support. 
This quantitative version of FUP yields an explicit essential spectral gap for convex co‑compact hyperbolic 3‑manifolds arising from  quasi‑Fuchsian groups, thereby refining the result of Tao \cite{Tao2025PhD}. Our result extends the earlier work of Jin–Zhang \cite{MR4081109}  to higher dimensions. \end{abstract}

\maketitle

\section{Introduction}
In this paper, we investigate the {fractal uncertainty principle} (FUP) in higher dimensions in a quantitative form and apply it to the spectral gap problem for certain hyperbolic manifolds.
Roughly speaking, the FUP states that no function can be simultaneously localized in both position and frequency near a fractal set. This principle has a number of striking applications in quantum chaos.

\subsection{Main results}
As the first main result, we establish the following theorem, which provides a quantitative version  of Cohen's result \cite[{Theorem 1.1}]{MR4927737}.

\begin{theorem}[Quantitative FUP]\label{thm:qfup}
      Let $d\ge 1$, $\nu\in (0, 1/3),\;h\in (0,1/100)$. Suppose\\
(1) $\mathbf X\subset[-1,1]^d$ is $\nu$-porous on balls from scales $h$ to $1$.\\ 
(2) $\mathbf Y\subset [-h^{-1},h^{-1}]^d$ is $\nu$-porous on lines from scales $1$ to $h^{-1}$.\\
Then there exists some constant $C=C(\nu,d)>0$ independent of $h$ and $f$  such that
\begin{equation}\label{eq:fup}\| f\mathbf 1_\mathbf X\|_2\le C h^{\beta}\|f\|_2\end{equation} 
holds for any $f\in L^2(\mathbb R^d)$ with $\supp \hat f\subset \mathbf Y$, where
\begin{equation}\label{eq:beta}
    \beta=\beta(\nu,d)= \exp\left[-\exp\left((C_d\,\nu^{-2}|\log \nu|)^{(C_d\,\nu^{-1}|\log \nu|)}\right)\right]
\end{equation}
and $C_d>0$ is a large constant depending only on dimension $d$.  
\end{theorem}

The theorem provides a quantitative statement that any $L^2$ function Fourier localized near a fractal set cannot simultaneously be localized near another fractal set in physical space. A more symmetric formulation of \eqref{eq:fup} is given by
\begin{equation}
\label{eq:fup-sym}
    \|\mathbf 1_{\mathbf X}\mathcal{F}_h\mathbf 1_{\mathbf Y}\|_{L^2(\mathbb R^d)\to L^2(\mathbb R^d)}\leq Ch^\beta
\end{equation}
where  $\mathcal{F}_h$ is the unitary semiclassical Fourier transform defined as 
\[
    \mathcal{F}_h f(\xi)={(2\pi h)^{-d/2}}\int_{\mathbb{R}^d} e^{- i\,x\cdot\xi/h}\,f(x)dx \qquad f\in L^2(\mathbb R^d)\,,
\]
$\mathbf Y$ and $\mathbf X\subset [-1,1]^d$ are $\nu$-porous on lines and balls respectively, from scales $h$ to $1$, and constants $C,\beta>0$ depend only on the dimension $d$ and porosity parameter $\nu$; specifically,  
$\beta=\beta(\nu,d)$ is given by \eqref{eq:beta}. This is the right form that will be used in Section \ref{sec:Spectral_gap}. 

Here, we consider fractal sets satisfying  more general porous conditions rather than the classical Ahlfors–David  regularity.
In particular, we impose the stronger \emph{porosity-on-lines} condition in frequency space instead of the \emph{porosity-on-balls} condition, primarily due to the well-known {counterexample} of \emph{lines in $\mathbb R^2$}, which is porous on balls but not on lines.
For details, see Definition \ref{def:porosity} and the subsequent discussions.  Finally, the boundedness restriction on 
 $\mathbf X$ and $\mathbf Y$ can be relaxed using an almost orthogonality argument, as demonstrated in Dyatlov--Jin--Nonnenmacher \cite {MR4374954}.

We focus on how the exponent $\beta$ depends on the porosity parameter $\nu$ while temporarily setting aside its explicit dependence on the dimension $d$---a dependence that is also believed to be tractable. The expression for  $\beta=\beta(\nu)$  given by \eqref{eq:beta}   in Theorem \ref{thm:qfup} is certainly not sharp and may be improved in various ways.

\begin{remark}
    This result is \emph{implicit} in the work of Cohen \cite{MR4927737}, relying on a general partial result of Han--Schlag \cite{MR4085124}.  Han and Schlag obtained quantitative estimates for product Ahlfors–David regular sets and for generic sets involving damping functions, without imposing specific Fourier-space geometry; see Theorem \ref{thm:hsl1dampingtofup}.
\end{remark}

It is well-known that the FUP can be used to study open quantum chaos, in particular the existence of 
essential spectral gaps for the Laplacian on non-compact (convex co-compact) hyperbolic manifolds. 
Using the quantitative FUP---Theorem \ref{thm:qfup}, we give an \emph{explicit} essential spectral gap for quasi-Fuchsian hyperbolic 3-manifold, as follows.
\begin{theorem}\label{thm:Spectral}
    Let $M=\Gamma \backslash \mathbb{H}^3$ be a convex co-compact 3-dimensional hyperbolic manifold
    such that $\Gamma$ is a quasi-Fuchsian group. That is, the associated limit set $\Lambda=\Lambda(\Gamma)\subset \mathbb{S}^2$ is a quasi-circle with Hausdorff dimension
    $\operatorname{dim}_H\Lambda\in (1,2)$.
    Then the resolvent 
    \[
        R(\lambda)=\left(-\Delta-1-\lambda^2\right)^{-1}
    \]
admits a meromorphic extension as a family of operators: $L^2_{\operatorname{comp}}(X)\to L^2_{\operatorname{loc}}(X)$, from 
    $\operatorname{Im}\lambda>0$ to $\operatorname{Im} \lambda\geq -\tilde\beta+\epsilon$ for any $\epsilon>0$, with finitely many poles. 
    Here
$\tilde\beta$ is given explicitly by  
    \[
        \tilde\beta=\frac12\,\beta(\nu,2), \quad \nu=\frac{1}{10^8(10^{11}C_\mu^2)^{1/(\delta-1)}C_{\operatorname{arc}}^2} 
    \]
where $\beta(\nu,2)$ is given by \eqref{eq:beta} with $d=2$, and the constants $C_\mu$ and  $C_{\mathrm{arc}}$ are given in \eqref{eq:delta_regularity_of_Lambda} and \eqref{eq:three_point_condition} respectively.

Moreover, $R(\lambda)$ satisfies the following
cutoff estimates: for any $\chi\in C_c^\infty(M)$ and $\epsilon >0$, there exists some constant $C_\epsilon$ depending on $\epsilon$,  and some constant  $\tilde C_{\chi,\epsilon}$ depending on $\epsilon$ and $\chi$,  such that
\[\|\chi R(\lambda) \chi\|_{L^2(M)\rightarrow L^2(M)}\le \tilde C_{\chi,\epsilon}\, |\lambda|^{-1-2\min(0,\operatorname{Im} \lambda)+\epsilon}\,,
 \quad 
\operatorname{Im} \lambda\in [-\tilde\beta+\epsilon,1], \ |\operatorname{Re} \lambda|\ge C_\epsilon\,.\]

\end{theorem}
In the statement, $\nu$ is the porosity-on-lines parameter, which is determined by the quasi‑Fuchsian group $\Gamma$; the quantities $C_\mu$ and $C_{\mathrm{arc}}$  encode respectively the $\delta$-regularity condition and  the \emph{three-point condition} satisfied by the limit set $\Lambda$.
\begin{remark}
    The celebrated rigidity theorem for quasi‑Fuchsian groups due to Bowen
    \cite{Bowen1979Hausdorff}
 implies that the limit set $\Lambda$ 
    of a quasi-Fuchsian group is either a genuine circle or a quasi‑circle with  $\operatorname{dim}_H \Lambda\in (1,2)$. 
Since the spectral gap for Fuchsian groups (whose limit set is a  circle) is trivial as all resonances can be computed explicitly, see subsection \ref{subsec:fuch-Appendix-B}, Theorem \ref{thm:Spectral} in fact holds for \emph{all} quasi‑Fuchsian groups.  Note that Theorem \ref{thm:qfup} does not work for circles in higher dimensions.
\end{remark}

\begin{remark}
  The existence of an essential spectral gap for quasi‑Fuchsian $3$‑manifolds has already been established by Zhongkai Tao in his PhD thesis~ \cite[Corollary 4.7.4]{Tao2025PhD}; see also \cite{MR4930594} for the case $\delta<1$. However, his argument is quite general and does \emph{not} yield an explicit gap in terms of the geometric data of the group $\Gamma$.
\end{remark}

\subsection{A rough review of the literature}
The recent milestone work of Bourgain and Dyatlov \cite{bourgain2018spectral}  established an FUP for sets in $\mathbb R$ with a porosity condition, opening a wide field for studying various forms of FUP and their applications to problems in quantum chaos. The proof in \cite{bourgain2018spectral}, which deals with Ahlfors–David regular sets but implies a version for porous sets of Dyatlov--Jin \cite{MR3849286} due to a rough equivalence between the two concepts, relies heavily on the Beurling–Malliavin (BM) theorem.  
In contrast, Dyatlov and Jin \cite{MR3803716} proved an FUP with exponent $\beta > \frac12 - \delta$ using Dolgopyat’s method, offering a different approach. 

The one-dimensional FUP of Bourgain-Dyatlov has been successfully applied to study quantum systems whose underlying classical dynamics are chaotic.  By applying the FUP to fractal sets arising from chaotic dynamical systems, one can control high-frequency waves on such systems. Applications of the FUP currently fall into two main categories:
(1) obtaining lower bounds on the $L^2$-mass of eigenfunctions \cite{MR3849286,MR4374954}, establishing control for the Schr\"odinger equation, and proving exponential decay of damped waves \cite{MR4374954,MR3934848,MR4061399}, on compact negatively curved surfaces;
(2) proving  spectral gaps, Strichartz estimates and exponential decay of waves on non-compact hyperbolic surfaces \cite{dyatlov2016spectral,MR4073195,MR4081109,MR3894924,huang2025lossless}  and spectral gaps for obstacle scattering in dimension 2 \cite{MR4736525}.  See also \cite{MR3993411} for a very nice survey on this topic.

A natural subsequent question is whether analogous results hold in higher dimensions. The problem becomes more complicated, as even the basic tools for an FUP are not readily available.  In particular, no higher-dimensional analogue of the BM theorem was known prior to \cite{MR4927737}. Han and Schlag \cite{MR4085124} proved an FUP where the physical set is porous  and the Fourier support is a special Cartesian product of one-dimensional porous sets.  Their proof uses the Cartan technique to obtain $L^2$ localization and constructs damping functions from the product structure and the one-dimensional BM theorem. Cladek and Tao \cite{CladekTao} proved an FUP  for $\delta$-regular sets with $\delta$ near $d/2$ in odd dimensions, using additive energy estimate, extending the work of Dyatlov--Zahl \cite{dyatlov2016spectral}.

Recently, Backus, Leng, and Tao \cite{MR4930594}  proved a higher-dimensional FUP, extending \cite{MR3803716} via Dolgopyat's method for $\delta$-regular sets, under the conditions that $\delta<d/2$ and the two fractal sets are not orthogonal in a certain sense. They also investigated the spectral gap for the Laplacian on convex co-compact hyperbolic manifolds; see also \cite{Tao2025PhD}.  Almost simultaneously, Cohen \cite{MR4927737} established a remarkable FUP for sets exhibiting porosity along lines, using a quantitative higher-dimensional BM theorem.  

The proof of this BM theorem is divided into two parts: a plurisubharmonic (PSH) BM theorem and an analytic (A) BM theorem.  The A-BM theorem is proved similarly to the one-dimensional case using Bourgain’s idea based on Hörmander's $L^2$ theory, while the PSH-BM theorem is more difficult. The main novelty lies in modifying the weight using dyadic techniques so that the integral mean of its Hessian over lines admits a finite lower bound. This \emph{mean-on-lines} condition is new and central to the higher-dimensional BM problem. 
See also \cite{MR4884567} for another BM theorem in higher dimensions for weights which can be majorized by radial functions. However, this result does \emph{not} yield explicit constants because it relies on the non‑quantitative one‑dimensional BM theorem. It also seems not applicable to the higher‑dimensional FUP, since the weight we require \emph{cannot} be majorized by radial functions.
 
Combining Cohen's BM theorem with the damping criteria of Han–Schlag yields the $d$-dimensional FUP. Since every step in \cite{MR4927737} is constructive, one can track all constants to obtain a quantitative version.

The FUP of Cohen has recently  been employed by Kim and Miller \cite{kim2025semiclassical} to investigate the supports of semiclassical measures for Laplace eigenfunctions on compact hyperbolic $(d+1)$-manifold, extending the argument of Dyatlov and Jin \cite{MR3849286} to higher dimensions. 
Moreover, Athreya, Dyatlov and Miller \cite{MR4956615}  previously demonstrated that, on a $2d$-dimensional compact complex hyperbolic quotient manifold, the support of semiclassical measures must contain the cosphere bundle of a compact immersed totally geodesic complex submanifold. Their argument relies solely on the one-dimensional FUP of \cite{bourgain2018spectral} along the fast expanding/contracting directions. It is conjectured even in higher dimensions that semiclassical measures should still have \emph{full} support, although this currently seems far from reach.

The original proof of the one-dimensional BM theorem in \cite{bourgain2018spectral} proceeds by contradiction, so the constants are not explicit. Jin and Zhang \cite{MR4081109} used a different, weaker version of the one-dimensional BM theorem, assuming a Lipschitz condition on the Hilbert transform, to prove a quantitative FUP. This follows ideas from the interesting ``seventh proof" given by Mashreghi, Mazarov and Havin \cite{MR2241422}. 
 \cite{MR4081109} also addressed the spectral gap problem, applying their quantitative FUP to obtain an explicit essential spectral gap for convex co-compact hyperbolic surfaces when the Hausdorff dimension of the limit set is close to 1. Our main results, Theorem \ref{thm:qfup} and \ref{thm:Spectral}, can be viewed as higher-dimensional analogues of \cite[Theorem 1.2 and 1.3]{MR4081109}.

 Last but not least, discrete versions of the FUP have also been established. For instance, one‑dimensional cases were treated in \cite{MR3656519, MR4779150, han2024fractal}, while \cite{Cohen} studied a two‑dimensional discrete setting.
In \cite{MR3656519}, an FUP for arithmetic Cantor sets of dimension $\delta<1$ was used to obtain a substantial improvement of the spectral gap in a quantum open baker’s map model.
The works \cite{MR4779150, han2024fractal} investigated further FUP for random discrete Cantor sets with improved exponents.
For two‑dimensional discrete Cantor sets, \cite{Cohen} provided an equivalent condition for the FUP to hold, namely that the sets contain no pair of orthogonal lines.

\subsection{Structure of the paper} Section \ref{sec:qfup} is devoted to the proof of Theorem \ref{thm:qfup}, which provides a detailed revisiting of Cohen's proof \cite{MR4927737} with precise tracking of all constants. Section \ref{sec:Spectral_gap} contains the proof of Theorem \ref{thm:Spectral}, establishing an explicit essential spectral gap for quasi-Fuchsian hyperbolic 3-manifolds.
The proof proceeds by first establishing {line porosity for the limit set}. This property relies on the fact that the limit set is a Jordan curve satisfying the {three-point condition} defined in \eqref{eq:three_point_condition}, which prevents it from wandering arbitrarily far and then returning arbitrarily close to its starting point. We then follow the argument developed in \cite{dyatlov2016spectral,bourgain2018spectral,kim2025semiclassical} to deduce the essential spectral gap. Finally, we also give the explicit computations of FUP for circles and spectral gap for Fuchsian groups. 
Appendix \ref{sec:appA} collects several useful lemmas concerning porosity.

\subsection*{Notations} 
Throughout the paper, $C$ and $c$ denote large and small positive constants, respectively, whose values may change from lines to lines.  Both $C(\cdot)$ and $C_{(\cdot)}$ indicate the dependence of $C$ on the parameters in $(\cdot)$. We denote by $B_r(x)$ or $B(x,r)$ the ball in $\mathbb{R}^d$ with radius $r>0$ and center $x$. The Japanese bracket is denoted by $\langle x \rangle = (1+|x|^2)^{1/2}$, while $|x|$ and $|x|_1$ denote the Euclidean $\ell_2$ and $\ell_1$ norms in $\mathbb{R}^d$, respectively. The indicator function for a set $X$ is written as 
$\mathbf 1_{X}$. For any set $X\subset \mathbb R^d$ and $r\ge0$, denote the $r$-neighborhood of $X$ by $X(r)=X+B_r(0)$.

\section{Quantitative FUP in higher dimensions}\label{sec:qfup}
This section presents Cohen's fractal uncertainty principle (FUP) with explicit constants. In particular, we specify the exponent $\beta=\beta(\nu)>0$ of the scale parameter $h<1/100$ for porosity parameter $0<\nu<1/3$. We consider functions whose frequency support is localized on a fractal set that is $\nu$-porous on lines at scale $h^{-1}$.

We first introduce several notions of porosity.
\begin{definition}[Porosity]
\label{def:porosity}
(1) A set $\mathbf X\subset \mathbb R^d$ is called $\nu$-\emph{porous on balls} from scales $\alpha_0$ to $\alpha_1$ if for every ball $B$ of diameter $\alpha_0<R<\alpha_1$, there exists a point $x\in B$ such that $B_{\nu R}(x) \cap \mathbf X=\emptyset$. \\
(2) A set $\mathbf X$ is called (stronger) $\nu$-\emph{porous on lines} from scales $\alpha_0$ to $\alpha_1$ if for every line segement $\tau$ of length $\alpha_0<R<\alpha_1$, there exists a point $x\in \tau$ such that $B_{\nu R}(x)\cap \mathbf X=\emptyset$. For $d=1$, the two definitions coincide.
\\
(3) A set $\mathbf X\subset [-1,1]^d$ is called (weaker) \emph{box porous} at scale $L\ge 3$ with depth $n$, where $L$ is an integer, if the following holds: let $\mathcal C_n$ be the collection of congruent cubes of side length $L^{-n}$
 obtained by partitioning $[-1,1]^d$. Then, for every cube $Q\in \mathcal C_n$ satisfying $Q\cap \mathbf X\neq \emptyset$, there exists a sub-cube $Q'\in\mathcal C_{n+1}$ such that $Q'\subset Q$ and $Q'\cap \mathbf X=\emptyset$.     
\end{definition}
\begin{example} The product of two one-dimensional $1/3$-Cantor sets is porous on lines, whereas the Sierpinski carpet is  porous only on balls. In $\mathbb R^1$, any Ahlfors–David $\delta$-regular set with $\delta<1$ is porous in any way above, but this fails in general in higher dimensions.
\end{example}
\begin{example}
A line in $\mathbb R^d$ ($d\ge 2$)  is porous on balls but not on lines. The porosity-on-lines condition for the Fourier support is \emph{essential} for the FUP; a counterexample in $\mathbb R^2$  is given by $\mathbf X=\{(t,0): t\in \mathbb R\}, \mathbf Y=\{(0,t): t\in \mathbb R\}$ equipped with standard Lebesgue measures, which satisfy $\widehat{\mu_{\mathbf X}}=\mu_{\mathbf Y}$.
\end{example}
\begin{example}\label{ex:circle}
A circle in  $\mathbb R^d$ ($d\ge 2$) is porous on lines for scale $\alpha_0>0$ with porosity paramter $\nu=\nu(\alpha_0)\rightarrow 0$ as $\alpha_0\rightarrow 0$. Thus, porosity-on-lines does not hold uniformly for $\alpha_0\rightarrow 0$, which indicates that FUP for the circle cannot follow from Cohen's theorem. This can also be seen from the fact that Proposition \ref{prop:Xis-line-porous} does not make sense,  since the constant $C_{\operatorname{out}}$ in Lemma \ref{lem:non-concentration} blows up when the limit set is a circle, whose Hausdorff dimension is $\delta=1$.  We will discuss more details in subsection \ref{subsec:fuch-Appendix-B}.
\end{example}

\subsection{From damping condition to FUP}
We first introduce a damping property of fractal sets, which is crucial for establishing higher-dimensional FUP.  This  was first observed by Han--Schlag \cite{MR4085124}.
\begin{definition}[Damping functions]
(1)  A set $\mathbf{Y}\subset \mathbb{R}^d$ is said to admit a Cohen's $l_2$ {damping function} with parameters $c_1,c_2,c_3,\alpha\in (0,1)$, if there exists a function $\psi\in L^2(\mathbb{R}^d)$ satisfying
    \begin{alignat}{2}
        \operatorname{supp} \hat{\psi}&\subset  B(0,c_1)\,,\nonumber\\
        \|\psi\|_{L^2(B(0,1))}&\geq  c_2\,, \nonumber\\
             |\psi(x)|&\leq \langle x\rangle^{-d}\,, \quad &\text{for all } x\in \mathbb{R}^d\,, \label{eq:damping(-d)decay}\\
             |\psi(x)|&\leq \exp\left(-c_3\frac{|x|}{\log(2+|x|)^\alpha}\right)\,, \quad &\text{for all } x\in\mathbf{Y}\,. \label{eq:dampingY}
    \end{alignat}  
(2) A set $\mathbf Y \subset \mathbb R^d$ is said to admit a Han--Schlag's $l_1$ {damping function} with parameters $c_1, c_2, c_3, \alpha\in (0,1)$, if there exists a function $\psi\in L^2(\mathbb R^d)$ satisfying
 \begin{alignat*}{2}
\supp \hat{\psi}&\subset[-c_1,c_1]^d\,,\\
\|\psi\|_{L^2([-1,1]^d)}&\ge c_2\,,\\
|\psi(x)|&\le \langle x\rangle^{-d}\,, \quad &\text{for all } x\in \mathbb{R}^d\,,\\
|\psi(
x
)|&\le \exp\left({-c_3\frac{|x|_1}{(\log (2+|x|_1))^\alpha}}\right)\,, \quad &\text{for all } x\in\mathbf Y\,, 
   \end{alignat*}  
where $|x|_1$ is the $l_1$ norm of $x$ in $\mathbb R^d$. 
\end{definition}

Han--Schlag \cite[Theorem 5.1]{MR4085124} established the following criterion for the  FUP in terms of box porosity and $l_1$ damping. 
\begin{theorem}[\cite{MR4085124} From $l_1$ damping to FUP]
\label{thm:hsl1dampingtofup}
Suppose\\
(1) $\mathbf X\subset[-1,1]^d$ is box porous at scale $L\ge 3$ with depth $n$, for all $n\ge 0$ with $L^{n+1}\le N$ and $N\gg 1$.\\
(2) $\mathbf Y\subset[-N,N]^d$ is such that for all $n\ge 0$ with $L^{n+1} \le N$ one has that for all $\eta\in[-NL^{-n}-3,NL^{-n}+3]^d$ the set 
\[L^{-n} \mathbf Y+[-4,4]^d+\eta\]
admits an $l_1$ damping function with parameters $c_1=(2L)^{-1}$ and $c_2, c_3, \alpha\in(0,1)$.\\
Assume that $0<c_3<c_3^*(d)\ll 1$. 
Then the FUP inequality 
\[\|f\mathbf 1_\mathbf X\|_2\le C N^{-\beta}\|f\|_2\]
holds for any $f\in L^2(\mathbb R^d)$ with $\supp \hat f\subset \mathbf Y$ and 
$N\ge N_0\gg 1$, where $\beta$ is given by 
\begin{equation}
\label{eq:hs-beta}
\beta
=-\frac{\log{(1-\gamma_0(T)/2)}}{T\log L} \qquad \text{with} \quad \ 
\gamma_0(T)=\frac{1-c_\varphi^2/ L^{2(T-1)}}{2C_*^2}\,, 
\end{equation}
\begin{equation}\label{eq:definition of gamma_0T and T_0 in Han-Schlag's result}
    \quad T\ge T_0=\left\lceil\frac{\log(2C_\varphi C_*^2+\sqrt{4C_\varphi^2C_*^4+c_\varphi^2})}{\log L}\right\rceil\,,
\end{equation}
\begin{equation}\label{eq:definition of C* in Han-Schlag's result}
    C_*=\exp{\left[c_3(R_1+2)(\log(1+R_1))^{-\alpha}/2\right]} \,,
\end{equation}
\begin{equation}
\begin{split}
\label{eq:definition of R1 in Han-Schlag's result}
   R_1=\max\Bigg\{\left(\frac{2d}{c_3}\right)^2,\exp{\left[\left(\frac{16\pi C_d}{c_3}\right)^{\frac{1}{1-\alpha}}\right]},\exp{\left[4^{\frac{1}{1-\alpha}}\right]}, \qquad \\
   \left[\frac{(d |\log c_1|)^d}{c_3}\right]^8, \left[4\log\left(\frac{2C_d}{c_2^2}\right)\right]^2,(8d)^4\Bigg\}\,.
\end{split}
\end{equation}
Here $C_\varphi, c_{\varphi}$ 
depend only on the Schwartz function $\varphi$ chosen in \cite{MR4085124}, which depends actually only on dimension $d$ if $\varphi$ is radial.
\end{theorem}

Taking $L=\lceil \sqrt d/\nu\rceil$ and $N=\lceil h^{-1}\rceil$, and noting that  $\nu$-porosity implies box-porosity (Lemma \ref{lem:compare_porosity}), we obtain another criterion for the quantitative FUP in terms of $\nu$-porosity and $l_2$ damping, which is implicit in Cohen \cite[Theorem 1.6]{MR4927737}, .
\begin{theorem}[\cite{MR4927737} From $l_2$ damping to FUP]\label{thm:Cohen'sl2dampingtoFUP}
Suppose\\
(1) $\mathbf X\subset[-1,1]^d$ is $\nu$-porous on balls from scales $h$ to $1$.\\
(2) $\mathbf Y\subset[-h^{-1},h^{-1}]^d$ satisfies: There exists $c_2,c_3,\alpha\in(0,1)$ such that for all $h<s<1$ and $\eta\in[-h^{-1}s-5,h^{-1}s+5]^d$, the set 
\[s\mathbf Y+[-4,4]^d+\eta\]
admits an $l^2$ damping function with parameters $c_1=\frac{\nu}{20\sqrt d}$ and $c_2,c_3,\alpha$.\\
Then for $\beta>0$ and $C>0$ independent of $h$ as in Theorem \ref{thm:hsl1dampingtofup} with $L=\lceil \sqrt d/\nu\rceil$, \[\|f\mathbf 1_\mathbf X\|_2\le C h^{\beta}\|f\|_2\]
 holds for any $f\in L^2(\mathbb R^d)$ with $supp \hat f\subset \mathbf Y$.  
\end{theorem}
Note that the $l_1$ damping  with $l_1$ Euclidean norm $|\cdot|_1$ is \emph{only} used in this subsection.  In the remainder of the paper, damping refers exclusively to the standard $l_2$ damping.

\subsection{From porosity to damping condition}
With the damping criteria in Theorem \ref{thm:Cohen'sl2dampingtoFUP}, it suffices to prove quantitatively that every $\nu$-porosity-on-lines set must admit a nice damping function. To prove this, we need the following quantitative BM theorem due to Cohen \cite[Theorem 1.4]{MR4927737}
\begin{theorem}[\cite{MR4927737} QBM] \label{thm:QBM}
    Let \(\omega:\mathbb{R}^{d} \to \mathbb{R}_{\leq 0}\) be a weight satisfying
    \begin{align}
        \omega(x) &= 0 \qquad \text{for } |x| \leq 2, \nonumber\\
        |D^{a}\omega(x)| &\leq C_{\mathrm{reg}} \langle x \rangle^{1-a} \qquad \text{for } 0 \leq a \leq 3, \label{cond:reg} \\
        \int_{0}^{\infty} \frac{G^{*}(r)}{1+r^{2}} \, dr &\leq C_{\mathrm{gr}}, \label{cond:gr}
    \end{align}
where 
\[G(x)=\int_{1/2}^2 |\omega(sx)| ds,\quad G^*(r)=\sup_{|x|=r} G(x).\]
Then for any \(\sigma > 0\), there exists a function \(f \in L^{2}(\mathbb{R}^{d})\) such that
    \begin{align}
        \operatorname{supp} \hat{f} &\subset B_{\sigma}, \nonumber \\
        |f(x)| &\geq \frac{1}{2} \qquad \text{for all } x \in B_{r_{\min}}, \nonumber \\
        |f(x)| &\leq C e^{c\sigma \omega(x)} \qquad \text{for all } x \in \mathbb{R}^{d}\,,\label{eq:BMcontrol-f-by-omega}\\
\end{align}
where
\begin{align}
c &= \frac{1}{C_d\max(C_{\mathrm{reg}}, C_{\mathrm{gr}})} \nonumber\\
    C &= C_d \, \max(\sigma^{-C_d}, e^{3\sigma}) \nonumber \\
 r_{\min} &= \frac{1}{C_d} \, \min(\sigma, \sigma^{-1})   \hskip50mm   \nonumber
\end{align}
and $C_d$ is a large constant  depending only on the dimension $d$.
\end{theorem}
This is a deep theorem about the existence of the pluri-subharmonic functions and analytic continuation. There is surely still room for the constants here  to be improved in various ways, but it is enough for us to get an explicit damping property and finally an explicit exponent for FUP using Theorem \ref{thm:Cohen'sl2dampingtoFUP}. We first track the constants from porosity to damping in \cite[Proposition 1.7]{MR4927737}.
\begin{theorem}[Porosity to damping]\label{thm:Porous-to-damping}
    Let $\mathbf{Y} \subset[-3h^{-1},3h^{-1}]^d$ be $\nu$-porous on lines from scales $\mu>1$ to $h^{-1}$ with $\nu<1/3$ and $h<1/100$. Then there exists a constant $C_d$ depending only on $d$, and some constant
\[\alpha>\alpha(\nu)=1-\frac{\nu}{C_d\log|\nu|}<1\]
 such that for any $0<c_1<1$, $\mathbf Y$ admits an $l_2$ damping function with parameter $\alpha, c_1$ and 
    \[ \ c_2=\frac{c_1^{C_{d}}}{C_d\,\exp{\big( \frac{c_1 \nu\mu}{C_d|\log\nu|}\big)}}
    \quad c_3=\frac{c_1\nu}{C_d\,|\log \nu|}.\]
\end{theorem}

\begin{proof}
    This is a repetition of Cohen, we write every constant explicitly in $\nu$. Note that in the proof we will use $C_d$ to denote a very large constant depending only on dimension $d$, whose value may differ at each occurrence.
    
Now, as a first step, we want to construct a good weight $\omega$ that satisfies the regularity and growth condition in
BM Theorem \ref{thm:QBM}, which however is also controlled on the fractal set $\mathbf Y$ by the damping exponent in \eqref{eq:dampingY} as 
\[\omega(x)\le -c \frac{|x|}{(\log(2+|x|))^\alpha}+C\,, \quad x\in \mathbf Y \quad \text{with} \  \alpha=1-\frac{\nu}{C_d\log|\nu|}\,.\]

For all \(k\geq 1\), we denote the dyadic annuli by \(A_{k}=\{x\in\mathbb{R}^{d}\,:\,2^{k}\leq|x|\leq 2^{k+1}\}\) and take 
\[
    W_k=\frac{2^k}{k^s}
\]
where $s=0.2$ is a uniform parameter.  The choice of $\alpha$ and $s$ is to assure the regularity and growth condition \eqref{cond:reg} and \eqref{cond:gr} in BM theorem for the weight to be constructed. This will be seen later. 

Consider the collection  \(\mathcal{Q}_{k}=\{ Q\} \) of cubes with side-length $W_k$ 
defined as follows:
\[
    \mathcal{Q}_k=\left\{ Q =\prod_{j=1}^d \left(W_k\frac{a_j}{3},W_k(\frac{a_j}{3}+1)\right):a_j\in \mathbb{Z}, \ Q \cap A_k\neq \emptyset\right\}.
\]
Then we have a covering of finitely overlapping cubes for the dyadic annuli:
    \begin{align*}
        A_k&\subset \bigcup_{ Q \in \mathcal{Q}_k}  Q/2 ,\\
        \bigcup_{ Q \in \mathcal{Q}_k} Q &\subset
        \{x\in\mathbb{R}^{d}\,:\,2^{k-1}\leq|x|\leq 2^{k+2}\},\quad k\geq k_0>\lceil (2d^{1/2})^{1/s}\rceil,
    \end{align*}
where $Q/2$ is the scaling cube with same center as and half length of $Q$ and $k_0$ is a large integer to be chosen later. Note that each point in $\mathbb{R}^d$ intersects at most $C=3^d$ number of cubes in $\mathcal{Q}_k$ for fixed $k$.

For the control of weight function $\omega$ on $\mathbf Y$, we only need to consider the effective cubes intersecting with $\mathbf Y$. 
Let
\begin{align*}
\mathcal{S}_{\mathbf{Y},k} &= \{ Q \in\mathcal{Q}_{k}\,:\, Q \cap(\mathbf{Y}\cap A_{k})\neq\emptyset\}, \\
\mathbf{Y}_{k} &= \bigcup_{ Q \in\mathcal{S}_{\mathbf{Y},k}} Q\,.
\end{align*} We will need to construct smooth functions supported on these effective cubes and sets.
Let \(\{\eta_{ Q }\}_{ Q \in\mathcal{Q}_k}\) be a collection of bump functions satisfying
\begin{align}
 \eta_Q &\in[0,1] \qquad \quad
     \supp \eta_{ Q } \subset Q\,,
     \nonumber\\
\eta_Q(x) &=1 \qquad \qquad \quad\forall\,x \in Q/2, \nonumber\\
\|D^{a}\eta_{ Q }\|_{\infty} &\leq C_{d,a}\,W_{k}^{-a} \quad \ \,   \forall\, a \geq 0, \label{eq:6.3 in Cohen's paper} \\
\sum_{ Q \in\mathcal{Q}_{k}}\eta_{ Q }(x) &\in[1,C_d] \qquad \quad \forall\, x\in A_{k}. \label{eq:partition}
\end{align}
Actually, we can construct these $\eta_Q$ by dilations and translations as $\eta_Q=\eta(W_k^{-1}(x-c_Q))$ where $\eta$ is a fixed radial bump function and $c_Q$ is the center of $Q$. Now we can construct the negative dyadic $C_c^\infty$ functions supported on the effective set $\mathbf Y_k$ with height $\frac{2^k}{k^\alpha}$ by 
\begin{equation*}
\omega_{k}=-\frac{2^{k}}{k^{\alpha}}\sum_{ Q \in\mathcal{S}_{\mathbf{Y},k}}\eta_{ Q } \label{eq:weight}
\end{equation*}
{with} support 
\[\operatorname{supp}\omega_{k} \subset  \mathbf Y_k\subset \{x\in\mathbb{R}^{d}\,:\,2^{k-1}\leq|x|\leq 2^{k+2}\}  \qquad \forall \, k\geq k_0\,.\]  
Note that from \eqref{eq:partition}
\begin{equation}\label{eq:6.6 in Cohen}
    \omega_{k}(x) \le  
 -\frac{2^{k}}{k^{\alpha}} \quad \text{for } x \in \mathbf{Y} \cap A_{k}.
\end{equation}

Now we set the negative $C^\infty$
weight function 
\begin{equation}\label{eq:omega}
    \omega = \sum_{k \geq k_{0}} \omega_{k},
\end{equation}
supported on $B_{2^{k_0-1}}^{\complement}$, where \(k_{0}=k_0(d,\mu) \geq 2\) is the smallest integer such that
\begin{equation}\label{eq:Choice-of-k0}
    W_{k_{0}} > \frac{\mu}{2\sqrt d}\;\text{ and }\; k_0^s> 10d.
\end{equation}
Actually, we can choose $k_0$ such that
\[W_{k_0}=\max\left\{\frac{\mu}{2\sqrt d}, \frac{2^{[(10d)^{1/s}]}}{10d}\right\}+1. \]
The choice of $k_0$ can be seem from \eqref{eq:YkinY+B} and the condition in Lemma \ref{lem:porous_transf}. Notice that \(\omega(x) = 0\) for \(|x| \leq 2^{k_0-1}\).  So the first vanishing condition in Theorem \ref{thm:QBM} is naturally satisfied. 
Besides, by \eqref{eq:omega}, \eqref{eq:6.6 in Cohen} and \eqref{eq:partition}, we have the control of $\omega$ on $Y$
\begin{equation*}
    \omega(x) \leq -\frac{1}{20} \frac{|x|}{\bigl(\log(2 + |x|)\bigr)^{\alpha}} \quad \text{for } |x| > 2^{k_{0}} \text{ and } x \in \mathbf{Y}, 
\end{equation*}
\begin{equation}\label{eq:omegacontrol_bydamping}
    \omega(x) \leq -\frac{1}{20} \frac{|x|}{\bigl(\log(2 + |x|)\bigr)^{\alpha}} + C(\mu,d) \quad \text{for all } x \in \mathbf{Y},
\end{equation}
where $C(\mu,d)$ can be taken as
\[
    C(\mu,d)=\mu+C_d.
\]
Actually, it suffices to satisfy
\[C(\mu,d)\ge \sup_{|x|\le 2^{k_0} } \frac{|x|}{\bigl(\log(2 + |x|)\bigr)^{\alpha}}=\frac{2^{k_0}}{\bigl(\log(2 + 2^{k_0})\bigr)^{\alpha}}\le  \frac{\max\{\mu,C_d\}}{2\sqrt d\max\{\log \mu, C_d\}^{\alpha-s}}\,,\]
being aware of the support of $\omega$, and $\omega_k$ and
since we will fix $s=0.2$ and $\alpha\in (0.9,1)$ which is very near $1$ if $\nu$ is small, we can either take much smaller constant
\[C(\mu,d)=\frac 1{C_d}\mu/(\log\mu)^{0.7}.\]
However, the log improvement will not contribute in the dominant (constant) term in the quantitative FUP. 

Now we begin to check the regularity and growth condition of $\omega$ in the BM Theorem \ref{thm:QBM}. 
We first establish the regularity property \eqref{cond:reg} for $\omega$. For any \(a \geq 0\), \(k \geq 1\), we have
\begin{equation}\label{eq:cre-from-porous}
    |D^{a}\omega_{k}| \leq C_{a,d} \,W_{k}^{-a}\, 2^{k} \,k^{-\alpha} \sum_{{Q} \in \mathcal{S}_{\mathbf{Y},k}} \mathbf 1_{{Q}} \leq C_{a,d}\,2^{(1-a)k}\, k^{as - \alpha}\, \mathbf 1_{\mathbf{Y}_{k}}, 
\end{equation}
where we use \eqref{eq:6.3 in Cohen's paper} for the first inequality and finite overlapping \eqref{eq:partition} of the cubes in \(\mathcal{Q}\) for the second inequality. As long as \(3s < \alpha\), \(\omega\) satisfies the regularity condition \eqref{cond:reg} with a constant \(C_{\mathrm{reg}}=C_d\) that depends only on the dimension $d$ (and the fixed radial bump function $\eta$).

Now we check the growth condition \eqref{cond:gr}. Note that 
\begin{equation}\label{eq:YkinY+B}
    \mathbf{Y}_{k} \subset (\mathbf{Y} \cap A_{k}) + B(0,2W_k\sqrt{d})\,.
\end{equation}
Since $\mathbf{Y}\subset [-3h^{-1},3h^{-1}]^{d}$, 
we can assume $2^{k}(1-\frac1{10\sqrt d})\leq 3h^{-1}\sqrt{d}$ since $k\ge k_0$, see \eqref{eq:Choice-of-k0}, otherwise $\mathbf{Y}_{k}$ will be an empty set. 
Since $k\ge k_0$  and $\mu<2W_k\sqrt d<h^{-1}$,  we can now apply Lemma \ref{lem:porous_transf} with parameter
\[
    \nu'=\nu/2,\quad \alpha_0=\mu,\quad \alpha_1=h^{-1},\quad r=2W_k\sqrt{d}
\]
to deduce that \(\mathbf{Y}_{k}\) is \(\nu/2\)-porous on lines from scales \(4W_{k}\sqrt{d}/\nu\) to $h^{-1}$ (this is a vacuous statement if $4W_{k}\sqrt{d}/\nu\geq h^{-1}$). 

We first claim that 
\begin{equation}\label{eq:measure-of-Ykcapline}
    |\mathbf{Y}_{k} \cap \ell|\leq C_d\,2^{k}\,k^{-s\gamma} \quad \text{ for all lines } \ell 
\end{equation}
for fixed exponent
\begin{equation}\label{eq:gamma}
    \gamma=\frac{\nu}{C_d|\log \nu|}<1\,.
\end{equation}
Actually, taking $R=\frac3{10}\frac{2^{k}}{\sqrt{d}}<\alpha_1=h^{-1}$,  then when  $4W_{k}\sqrt{d}/\nu\geq R$, we have $k^{-s}/\nu\geq \frac{1}{100d}$, and the trivial estimate 
\[
    |\mathbf{Y}_{k} \cap \ell|\leq 2^{k+2}\leq 2^{k+2}(k^{-s}/\nu)^{\gamma}C_d^{\gamma}\leq C_d\nu^{-\gamma}2^kk^{-s\gamma}\,.
\]
When $4W_{k}\sqrt{d}/\nu<R$, 
we can decompose $\ell\cap A_k$ into $C\sqrt d$ many line segments of 
length $R$.
Then applying Lemma \ref{lem:measurelineintersection} to  each segment and $\mathbf Y_k$, we obtain
\[
    |\mathbf{Y}_{k} \cap \ell|\leq C_d2^{k+1}\left(\frac{20W_kd}{2^{k}\nu}\right)^{\gamma}\leq C_d\,\nu^{-\gamma}\,2^{k}\,k^{-s\gamma}.
\]
This proves that 
\[
    |\mathbf{Y}_{k} \cap \ell|\leq C_d\,\nu^{-\gamma}\,2^{k}\,k^{-s\gamma} \quad \text{ for all lines } \ell 
    \]
which reduces to the claimed estimate \eqref{eq:measure-of-Ykcapline} by further noting that 
\[
    \nu^{-\gamma}=e^{\nu/C_d}\in (1, e^{\frac 1{C_d}}).
\]

Using \eqref{eq:measure-of-Ykcapline} for \(\ell = \{tx : t \in \mathbb{R}\}\), a line through the origin with direction $x\in \mathbb S^{d-1}$, we see
\begin{equation}
\label{eq:omega_k-integral-on-lines}
    2^{-k} \int_{0}^{\infty} |\omega_{k}(tx)| \, dt \leq \|\omega_k\|_{L^\infty} |\mathbf{Y}_{k} \cap \ell| \leq C_d2^{k} \, k^{-(\alpha + s\gamma)}\,.
\end{equation}
Let \(G^{*}(r)\) be the growth function  in \eqref{cond:gr} with \(r \in [2^{k}, 2^{k+1})\) for some $k\ge 2$ (note that $G^*(r)=0$ for $r\leq 4$). Using \eqref{eq:omega_k-integral-on-lines} we have the pointwise bound 
\begin{align*}
    G^{*}(r) &\leq C_d \sup_{x \in \mathbb S^{d-1}} 2^{-k} \int_{2^{k-1}}^{2^{k+2}} |\omega(tx)| \, dt \nonumber \\
    &\leq C_d \sup_{x \in \mathbb S^{d-1}} 2^{-k} \sum_{k-3 \leq j \leq k+3} \int_{0}^{\infty} |\omega_{j}(tx)| \, dt \nonumber \\
    &\leq C_d\frac{r}{(\log(2+r))^{\alpha + s\gamma}}\,.
\end{align*}

As long as \(\alpha + s\gamma > 1\), the growth condition \eqref{cond:gr} is satisfied with a constant $C_{\operatorname{gr}}$ which can be computed as 
\begin{equation}\label{eq:cgr-from-porous}
    C_{\operatorname{gr}}=C_d\int_{3}^\infty \frac{1}{r(\log r)^{\alpha+s\gamma}} dr=\frac{C_d}{(\alpha+s\gamma-1)(\log 3)^{\alpha+s\gamma-1}}\sim C_d/\gamma\,.
\end{equation}
Note we have fixed \(s = 0.2\) universally. We can choose \(\alpha \ge 1 - 0.1\gamma\), and then \(-\alpha + 3s \le -0.3\) and \(\alpha + s\gamma \ge 1 + 0.1\gamma>1\).


By applying QBM Theorem \ref{thm:QBM} with spectral radius \(\sigma < 1\) on $\omega$ with regularity and growth  constant from \eqref{eq:cre-from-porous} and \eqref{eq:cgr-from-porous} 
\[C_{\operatorname{reg}}=C_d, \qquad C_{\operatorname{gr}}={C_d}/{\gamma},\]
 we can obtain that, for any $0<\sigma<1$, there exists a function \(f \in L^{2}(\mathbb{R}^{d})\) satisfying
\begin{align*}
    &\operatorname{supp} \hat{f} \subset {B}_{\sigma/2}, \\
    &|f(x)| \geq \frac{1}{2} \quad \text{for } |x| \leq \frac{\sigma}{C_d}, \\
    &|f(x)| \leq C_d\sigma^{-C_{d}} \,\exp(\sigma \gamma\,\mu/C_d)\,\exp\left( -\sigma \frac{\gamma}{C_d} \frac{|x|}{(\log(2+|x|))^{\alpha}} \right) \quad \text{for } x \in \mathbf{Y}, \\
    &|f(x)| \leq C_d\sigma^{-C_{d}} \quad \text{for } x \in \mathbb{R}^{d}.
\end{align*}
The last two estimates are derived from  \eqref{eq:omegacontrol_bydamping} and \eqref{eq:BMcontrol-f-by-omega} ($\omega\le 0$).
In order to satisfy the $(-d)$-decay \eqref{eq:damping(-d)decay} in $\mathbb R^d$, we multiply $f$ with a Schwartz function. 
Now let \(\varphi_0 : \mathbb{R}^d \to \mathbb{R}\) be a fixed Schwartz function with supp \(\hat{\varphi}_0 \subset B(0,1)\) and $\varphi_0(0)=1$. Then there exists some $r_d\in (0,1/2)$ such that $\varphi_0(x)\geq 1/2$ for all $|x|<r_d$. Let \(\varphi(x) := \varphi_0(\sigma r_dx/10)\) so supp \(\hat{\varphi} \subset B(0,{r_d\sigma/10})\) and \(\int \hat{\varphi} = 1\). Let  
\( f_1 = f\varphi \). Then \( \hat{f}_1 = \hat{f} * \hat{\varphi} \) and
\begin{align*}
    &\operatorname{supp} \hat{f_1} \subset {B}_{\sigma}, \\
    &|f_1(x)| \geq \frac{1}{2} \quad \text{for } |x| \leq \frac{\sigma}{C_d}, \\
    &|f_1(x)| \leq C_d\,\sigma^{-C_{d}}\,\exp( \sigma \gamma\,\mu/C_d)\, \exp\left( -\sigma \frac{\gamma}{C_d} \frac{|x|}{(\log(2+|x|))^{\alpha}} \right) \quad \text{for } x \in \mathbf{Y}, \\
    &|f_1(x)| \leq C_d\sigma^{-C_{d}}\langle x\rangle^{-d} \quad \text{for } x \in \mathbb{R}^{d}.
\end{align*}
The last equation follows from the trivial estimate \(|\varphi_0(x)| \leq C_d\langle x\rangle^{-d}\). Dividing through by the constant 
\[C_d\,\sigma^{-C_{d}}\,\exp(\sigma \gamma\,\mu/C_d)\]
and by \eqref{eq:gamma}, we obtain the $l_2$ damping function with parameters 
    \begin{align*}
        c_{1} = \sigma\,, \qquad
        c_{2} = \frac{\sigma^{C_{d}}}{C_d \exp(\sigma \gamma\,\mu/C_d)}=\frac{\sigma^{C_{d}}}{C_d \exp( \frac{\sigma \nu\mu}{C_d|\log\nu|})}\,, \qquad
        c_{3} = \frac{\gamma}{C_d}\sigma=\frac{\sigma \nu}{C_d|\log\nu|}\,.
    \end{align*}
  This completes the proof of the theorem.
\end{proof}

\subsection{Quantitative FUP}
By combining the results from the previous two subsections, we deduce  the FUP from porosity. Applying Theorem \ref{thm:Cohen'sl2dampingtoFUP} and  \ref{thm:Porous-to-damping} with the choices 
 \[\mu=10\sqrt d/\nu,\qquad c_1=\frac\nu{20\sqrt d}\,, \]
 we can now prove the quantitative FUP stated in Theorem \ref{thm:qfup}.

\begin{proof}[Proof of Theorem \ref{thm:qfup}]

As before, in the proof we will use $C_d$ to denote a very large constant depending only on dimension $d$, whose value may differ at each occurrence.
    
    Let the two fractal sets in the physical and Fourier spaces be
    \begin{itemize}
        \item  $\mathbf X\subset[-1,1]^d$ is $\nu$-porous on \emph{balls} from scales $h$ to $1$.
        \item $\mathbf Y\subset [-h^{-1},h^{-1}]^d$ is $\nu$-porous on \emph{lines} from scales $1$ to $h^{-1}$.
    \end{itemize}
    
    We \emph{claim} that, $\mathbf Y$ satisfies the damping condition in Theorem \ref{thm:Cohen'sl2dampingtoFUP}  with parameters
\begin{align}
        \label{eq:Parametr for Y in the proof of theorem QFUP}
            \alpha&=1-\frac{\nu}{10C_d\log|\nu|} \hskip16mm
            c_1=\frac{\nu}{20\sqrt d} \nonumber\\
            c_2&= \frac{\nu^{C_d}}{C_d\exp(\frac1{C_d}\nu/|\log \nu|)} \qquad
            c_3=\frac{\nu^2}{C_d|\log\nu|}\,.
        \end{align}
    Actually, for $s\in (h,10h]$, this claim is trivial since
    \[
        \mathbf Y_{s,\eta}:=s\mathbf Y+[-4,4]^d+\eta\subset [-30,30]^d\,.
    \]
For all $10h<s<1$ and $\eta\in[-h^{-1}s-5,h^{-1}s+5]^d$, the set $\mathbf Y_{s,\eta}$
    is $\nu/2$-porous on lines from scales $10\sqrt{d}/\nu$ to $h^{-1}s$ by Lemma \ref{lem:porous_transf},  contained in $[-3h^{-1}s,3h^{-1}s]^d$.
    So applying Theorem \ref{thm:Porous-to-damping} with $\mu=10\sqrt{d}/\nu$ and $h$ therein substituted by $h^{-1}s$, the set $\mathbf Y_{s,\eta}$ admits a damping functions with parameter $c_1=\frac{\nu}{20\sqrt{d}}$ and 
    \[
        \begin{aligned}
            &c_2\geq \frac{c_1^{C_d+2d}}{C_d\exp(\gamma c_1\mu/C_d)}\geq \frac{\nu^{C_d}}{C_d\exp(\frac1{C_d}\,\nu/|\log \nu|)}\\
            &c_3\geq \frac{\nu^2}{C_d|\log\nu|}\,, \qquad
            \alpha\ge1-\frac{\nu}{10C_d\log|\nu|}\,.
        \end{aligned} 
    \]
    This proves our claim.
    
    So we can apply Theorem \ref{thm:Cohen'sl2dampingtoFUP} 
    with damping parameters in \eqref{eq:Parametr for Y in the proof of theorem QFUP}
    to obtain that 
    \[
        \|f\mathbf 1_\mathbf X\|_2\le C h^{\beta}\|f\|_2
    \]
where using the formula of Han-Schlag,
    \[\beta=\beta(d, L,c_1,c_2,c_3,\alpha)=\beta(d,\lceil \sqrt d/\nu\rceil,c_2,c_3,\alpha)\] 
  given by \eqref{eq:hs-beta} in Theorem \ref{thm:hsl1dampingtofup}, and $C>0$ is a constant independent of $h$ and $f$.

    Next we calculate the exponent $\beta$ explicitly. We first calculate $R_1$ defined in \eqref{eq:definition of R1 in Han-Schlag's result}. Since we can choose $C_d$ large enough, so in the ‌maximization process we can ignore $(8d)^4$. 
Since $c_3$ and $1-\alpha$ are very small,  we can assume $\max\{c_3,1-\alpha\}<0.01$, and then we have
    \[
      \max\left\{ \exp{[4^{\frac{1}{1-\alpha}}]}, \left(\frac{2d}{c_3}\right)^2\right\} \leq 
        \exp{\left[\left(\frac{16\pi C_d}{c_3}\right)^{\frac{1}{1-\alpha}}\right]}\,.
    \]
Moreover we have 
         \begin{align*}
          \exp{\left[\left(\frac{16\pi C_d}{c_3}\right)^{\frac{1}{1-\alpha}}\right]}&\leq \exp\left[\left(C_d|\log \nu|\nu^{-2}\right)^{\frac{C_d|\log \nu|}{\nu}}\right],\\
            \left[4\log\left(\frac{2C_d}{c_2^2}\right)\right]^2&\leq16(\log 2C_d+2|\log c_2|)^2\leq C_d|\log\nu|^{2},\\
            \left[\frac{(d|\log c_1|)^d}{c_3}\right]^8&\leq c_3^{-8}C_d(1+\log(\nu^{-1}))\leq C_d\,\nu^{-17},
        \end{align*}
    so in the notation of \eqref{eq:definition of R1 in Han-Schlag's result} we can actually take 
    \[
        R_1=\exp\left(\left(C_d|\log \nu|\nu^{-2}\right)^{{C_d|\log \nu|}{\nu^{-1}}}\right).
    \]
    Then we can further choose 
    \[ 
        C_*=\exp(R_1+2)
    \]
    in \eqref{eq:definition of C* in Han-Schlag's result} and choose  in \eqref{eq:definition of gamma_0T and T_0 in Han-Schlag's result} and \eqref{eq:hs-beta} 
 \[T_0=\log(C_dC_*^2)\leq C_dR_1,\quad T=\lceil C_dR_1\rceil,\quad
            \gamma_0(T)\geq \frac{1-c_{\nu,d}}{2C_*^2}\,.
    \]
So we can finally choose $\beta$ as 
    \[
        \begin{aligned}
        \beta&\geq \frac{-\log(1-\frac{1-c_{\nu,d}}{4C_*^2})} {T\log L}\geq 
        \frac{1-c_{\nu,d}}{4C_*^2T\log L}\geq \frac{1}{C_dC_*^2R_1|\log \nu| }\geq \frac{1}{C_d\exp(3R_1)}\\
        &\geq \frac 1{C_d}\exp\left[-3\exp\left((C_d\nu^{-2}|\log \nu|)^{C_d|\log \nu|\nu^{-1}}\right)\right]\,.
        \end{aligned}
    \]
    This completes the proof.
\end{proof}

\section{Spectral gap for quasi-Fuchsian 3-manifold}\label{sec:Spectral_gap}
In this section, we will prove Theorem \ref{thm:Spectral}.
We first recall the initial work of Dyatlov--Zahl \cite{dyatlov2016spectral}, who 
reduce the spectral gap problem to an FUP. 
In this section, we consider the specific case  $d=2$. Hereafter, the notation $\beta(\nu)$ will always refer to the constant formula $\beta(\nu, d)$ defined in \eqref{eq:beta} with $d=2$.

Let $\mathbb{H}^3$ be the three-dimensional hyperbolic space, 
$\Gamma \subset \operatorname{Iso}(\mathbb{H}^3)$ a Kleinian group,
$M = \Gamma \backslash \mathbb{H}^3$ the associated hyperbolic $3$-manifold, and 
$\Lambda = \Lambda(\Gamma) \subset \mathbb{S}^2$ the limit set of $\Gamma$. 
For $\alpha > 0$, define the $\alpha$-neighbourhood of $\Lambda$ on $\mathbb{S}^2$ by
\[
\Lambda(\alpha)=\{\, y \in \mathbb{S}^{2} \mid d(y,\Lambda) \leq \alpha \,\},
\]
where $d(y, y') = |y - y'|$ denotes the Euclidean distance in $\mathbb{R}^{3}$ restricted to $\mathbb{S}^{2}$.

Consider the operator $\mathcal{B}=\mathcal{B}_{\chi}(h):L^{2}(\mathbb{S}^{2})\to L^{2}(\mathbb{S}^{2})$ given by
\begin{equation}
\label{eq:b-chi}
\mathcal{B}_{\chi}(h)f(y)
\;=\;
(2\pi h)^{-1} \int_{\mathbb{S}^{2}} 
|y - y'|^{2i/h} \, \chi(y, y') \, f(y') \, dy',
\end{equation}
where $dy'$ is the standard volume form on $\mathbb{S}^{2}$ and 
$\chi \in C_c^{\infty}(\mathbb{S}_{\Delta}^{2})$ with
\[
\mathbb{S}_{\Delta}^{2} 
= \{\, (y, y') \in \mathbb{S}^{2} \times \mathbb{S}^{2} \mid y \neq y' \,\}.
\]
\begin{definition}
The limit set $\Lambda$ is said to satisfy the \emph{hyperbolic fractal uncertainty principle} with exponent $\beta>0$ if, for every $\varepsilon>0$, there exists $\rho\in(0,1)$ such that
\begin{equation}
\label{eq:hyperfup}
   \| \mathbf 1_{\Lambda(C_{1} h^{\rho})} \mathcal{B}_{\chi}(h) \mathbf 1_{\Lambda(C_{1} h^{\rho})} \|_{L^{2}(\mathbb{S}^{2})\to L^{2}(\mathbb{S}^{2})}\leq Ch^{\beta-\varepsilon},\quad h\in(0,1/100)  
\end{equation}
for any $h$-independent constant $C_{1}>0$ and any cutoff function $\chi$.  
Here the constant $C>0$ may depend on $C_{1}, \varepsilon, \rho, \chi$, and $\Lambda$, and $\mathcal{B}_{\chi}$ is the Fourier integral operator defined in \eqref{eq:b-chi}.
\end{definition}

\cite{dyatlov2016spectral} relates the existence of an essential spectral gap for convex co-compact hyperbolic manifolds to the validity of hyperbolic FUP. 
In light of \cite[Theorem 3]{dyatlov2016spectral}, Theorem \ref{thm:Spectral} follows immediately from the following proposition.
\begin{proposition}
\label{prop:hyperfup}
The limit set
 $\Lambda$ 
   satisfies the hyperbolic FUP \eqref{eq:hyperfup} with exponent 
   \begin{equation}\label{eq:beta-hyper}
  \beta=\frac 12\,\beta(\nu), \quad \nu=\frac{1}{10^8(10^{11}C_\mu^2)^{1/(\delta-1)}C_{\operatorname{arc}}^2}\,,
\end{equation}
where the constants $C_\mu$ and  $C_{\mathrm{arc}}$ are given in \eqref{eq:delta_regularity_of_Lambda} and \eqref{eq:three_point_condition} respectively.
\end{proposition}

By means of a partition of unity, we may assume that the support 
of $\chi$ is contained in $\Sigma\times \Sigma$, where $\Sigma$ is the region of $\mathbb{S}^2$ obtained by removing a spherical cap of solid angle $\pi/2$. 
Let $\pi:\mathbb{S}^2\to {\mathbb{R}^2}\cup \{\infty\}$ be the stereographic projection
sending $\Sigma$ to a subset of $B_{\mathbb R^2}(0,5)$. Then the desired FUP estimate \eqref{eq:hyperfup} with exponent \eqref{eq:beta-hyper} for $\mathcal{B}_{\chi}$ can be reduced to the following estimate for a new  Fourier integral operator ${B}_\chi:L^2(\mathbb{R}^2)\to L^2(\mathbb{R}^2)$ with a specific logarithmic phase.

\begin{proposition}
\label{prop:hyperFUPEuclidean}
Let $\beta$ be in \eqref{eq:beta-hyper}. 
For any $\epsilon>0$, there exists some $\rho\in(0,1)$ such that for any  $C_1>0$, there exists $C>0$, independent of $h$, 
such that
\begin{equation}\label{eq:hyperbolicFUP-euclidean}
    \| \mathbf 1_{X(C_{1}h^{\rho})} {B}_\chi(h) \mathbf 1_{X(C_{1}h^{\rho})} \|_{L^{2}(\mathbb{R}^2)\to L^{2}(\mathbb{R}^2)}\leq Ch^{\beta-\varepsilon},\quad h\in(0,1/100)  
\end{equation}
where 
    \begin{align}\label{eq:B_b}
X=&\pi(\Lambda)\cap B(0,5), 
\nonumber\\
   {B}_\chi(h)f(x)=&(2\pi h)^{-1}\int_{\mathbb{R}^2} e^{i\Phi(x,x')/h}\chi(x,x') f(x')dx',\\
    \Phi(x,x'):=&2\ln (|\pi^{-1}(x)-\pi^{-1}(x')|)=2\ln\left(\frac{2|x-x'|}{\sqrt{
    ({1+|x|^2})(1+|x'|^2)}}\right), \nonumber
    \end{align}
and $\chi(x,x')\in C_c^\infty(B(0,10)^2)$ is any cutoff supported away from the diagonal.
\end{proposition}

\subsection{Line porosity condition of $X$}
We begin by examining properties of the limit set $\Lambda = \Lambda(\Gamma)$ in our setting. 
Standard Patterson--Sullivan theory (see for example \cite[Theorem 7]{sullivan1979density}) implies that  $\Lambda\subset \mathbb{S}^2$ is $\delta$-\textit{regular}
for 
\[\delta=\operatorname{dim}_H(\Lambda)\in (1,2).\] That is, there exists a Borel measure $\mu$ on $\mathbb{S}^2$ and a constant 
$C_\mu>0$ such that
\begin{equation}\label{eq:delta_regularity_of_Lambda}
    \frac{1}{C_\mu}r^\delta\leq \mu(B_{\mathbb{S}^2}(x_0,r)\cap \Lambda)\leq C_\mu r^\delta,\quad 
    \forall \, x_0\in \Lambda,r\in [0,1]\,.
\end{equation}
This $\delta$-regular condition holds for limit sets of general convex co-compact hyperbolic manifolds. 
What distinguishes $\Lambda$ here is that it is a \textit{quasi‑circle}; i.e., the image of the unit circle under a 
$K$-quasiconformal map: $\widehat{\mathbb{C}} \to \widehat{\mathbb{C}}$, where we identify $\mathbb{S}^2$ with 
$\widehat{\mathbb{C}}$ (see \cite[Chapter 1.1]{gehring2012ubiquitous} or \cite[Chapter 2]{Ahlfors2006Lectures} for  definitions).  
Consequently, $\Lambda$ is homeomorphic to $S^1$ and satisfies the following \textit{three-point condition}:
there exists a constant $C_{\operatorname{arc}}$ such that
for every pair of points $z_1,z_2\in \Lambda$
\begin{equation}\label{eq:three_point_condition}
    \min_{j=1,2} \operatorname{diam}(\gamma_j)\leq C_{\operatorname{arc}} |z_1-z_2|
\end{equation}
where $\gamma_1,\gamma_2$ are the two distinct arcs of $\Lambda$ joining $z_1$ and $z_2$; see \cite[Theorem 2.2.5]{gehring2012ubiquitous} and also Figure \ref{fig:koch-snowflake} for an example.

\begin{remark}
    A concrete value for $C_{\mathrm{arc}}$ can be given as $C_{\mathrm{arc}} = 10 e^{8K}$. 
    Indeed, the discussion after \cite[Remark 2.2.6]{gehring2012ubiquitous} shows that if $\Lambda \subset \mathbb{C}$ 
    is the image of $\mathbb{S}^1 \subset \mathbb{C}$ under a $K$-quasiconformal map $\widehat{\mathbb{C}} \to \widehat{\mathbb{C}}$ 
    fixing $\infty$, then \eqref{eq:three_point_condition} holds with $C_{\mathrm{arc}} = 2 e^{8K}$ for the Euclidean 
    distance $d_{\mathbb{R}^2}$.  The inequality on the sphere follows by applying the stereographic projection.
\end{remark}

\begin{figure}[ht]
	\centering
	\begin{tikzpicture}[decoration=Koch snowflake, scale=1.2]
	\draw[thick, blue] decorate{ decorate{ decorate{ decorate{ 
					(0,0) -- (3,0) -- (1.5,-2.6) -- cycle 
	}}}};
	\end{tikzpicture}
	\caption{
		The primary example in our analysis is the celebrated Koch snowflake. 
This set is \((\log_3 4)\)-regular and clearly satisfies the three-point condition.
	}
	\label{fig:koch-snowflake}
\end{figure}
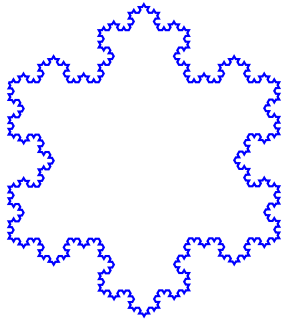

The main ingredient of this section is the following proposition.
\begin{proposition}\label{prop:Xis-line-porous}
    $X=\pi(\Lambda)\cap B(0,5)$ is  $\nu$-porous on lines from scales $\alpha_0=0$ to $\alpha_1>0$ 
     with 
\begin{equation}
\label{eq:nu-of-X}
\nu=\frac{1}{200C_{\mathrm{out}}\tilde{C}_{\operatorname{arc}}^{2}}\qquad \alpha_1=\frac{1}{10C_{\mathrm{out}}\tilde{C}_{\operatorname{arc}}}\,, 
\end{equation}
    where  $\tilde{C}_{\operatorname{arc}}$ is defined in \eqref{eq:tildeCmu and tildeCarc} and
    $C_{\mathrm{out}}$ is defined in \eqref{eq:definition-of-Cout}. 
  \end{proposition}
     More precisely, the theorem states that for any line segment $I\subset \mathbb{R}^2$ 
    of length $r\in [0,\alpha_1]$, 
    there exists a ball $B$ with radius 
    $\nu r$ centering at a point on $I$, such that $B\cap X=\emptyset$. 
  
We first interpret the $\delta$-regularity and three-point conditions \eqref{eq:delta_regularity_of_Lambda} and 
\eqref{eq:three_point_condition}
on $\Lambda$ into local Euclidean versions on $\pi(\Lambda)$ (by abuse of  notation we still use the same notations as on the sphere)
\begin{equation}\label{eq:Euclidean-delta-regularity}
    \frac{1}{\tilde{C}_\mu}r^\delta\leq \mu(B_{\mathbb{R}^2}(x_0,r)\cap \pi(\Lambda))\leq \tilde{C}_\mu r^\delta,\quad 
    \forall\, x_0\in \pi(\Lambda)\cap B(0,10),r\in [0,1]\,,
\end{equation}
\begin{equation}\label{eq:Euclidean-three-point-condition}
    \min_{j=1,2} \operatorname{diam}_{\mathbb{R}^2}(\pi(\gamma_j)\cap B(0,10))\leq \tilde{C}_{\operatorname{arc}} |z_1-z_2| \quad \forall\, z_1,z_2\in \pi(\Lambda)\cap B(0,10)\,,
\end{equation}
where the new constant $\tilde{C}_\mu$ and $\tilde{C}_{\operatorname{arc}}$ 
can be taken as 
\begin{equation}
\label{eq:tildeCmu and tildeCarc}
    \tilde{C}_\mu=20000C_\mu,\quad   \tilde{C}_{\operatorname{arc}}=200C_{\operatorname{arc}}\,.
\end{equation}

We begin the proof with the following lemma, which implies that $X$ cannot concentrate on a small neighborhood of a line, due to the $\delta$-regularity with $\delta>1$.
\begin{lemma}\label{lem:non-concentration}
There exists a constant $C_{\mathrm{out}}\geq 1$ such that for any point $x\in X$, any line $L\subset \mathbb{R}^2$  and any $r\in(0,1/2)$, we can find a point $y\in B(x,r)\cap \pi(\Lambda)$ such that
\[
\operatorname{dist}(y,L)\geq r/C_{\mathrm{out}}\,.
\]
We can actually take 
\begin{equation}\label{eq:definition-of-Cout}
    C_{\mathrm{out}}=2(4\cdot 5^\delta\tilde{C}_\mu^2)^{\frac{1}{\delta-1}}\,,
\end{equation}
where $\delta=\dim_H \Lambda\in(1,2)$ is the regularity parameter of $\Lambda$.
\end{lemma}
\begin{proof}
    The argument here is inspired by
    \cite[Lemma 2.1]{mattila2009ahlfors}. 
    For each $r'\in (0,r]$, we can invoke the standard covering lemma (see e.g. \cite[Theorem 2.1]{Mattila1995}), to find a collection of disjoint balls     
 $\{B(x_i,r')\}_{1\le i\le m}$,  such that  for all $i$
 \[x_i\in B(x,r)\cap \pi(\Lambda)\subset \bigcup_i B(x_i,5r)\,. \]

    Using the $\delta$-regularity condition \eqref{eq:Euclidean-delta-regularity} and the relation
    \[
    \begin{aligned}
        &\sum_{i} \mu(B(x_i,r')\cap \pi(\Lambda))\leq 
        \mu(B(x,2r)\cap \pi(\Lambda))\\
        &\mu(B(x,r)\cap \pi(\Lambda))
        \leq \sum_{i} \mu(B(x_i,5r')\cap \pi(\Lambda))\,,
    \end{aligned}
    \]
    we know the number $m$ of these balls, is finite, and satisfies 
    \begin{equation}
   \label{eq:mlowerbound}
        \frac{1}{5^\delta \tilde{C}_\mu^2}\left(\frac{r}{r'}\right)^\delta \leq m\leq \tilde{C}_\mu^2\left(\frac{r}{r'}\right)^\delta \,.
    \end{equation}
    
    We \emph{claim} that for sufficiently small $r'$, there must be some $y=x_{i_0}$ whose distance to the line $L$ is at least $\frac{\sqrt{3}}{2}r'$.
    Actually, by contradiction, if every $x_i$ has distance at most $\frac{\sqrt{3}}{2}r'$ to the line $L$, then the intersection of $B(x_i,r')$ and $L$ is an open interval of $L$ with length at least $r'$, and we observe that all such intervals are disjoint, and are contained in $B(x,2r)$. 
    This forces the number $m$ of intervals to satisfy 
    \[
        mr'\leq 4r, 
    \]
and thereby  implies a contradiction with \eqref{eq:mlowerbound} if we take 
    \[
        r'=\frac{0.99r}{(4\cdot 5^\delta\tilde{C}_\mu^2)^{\frac{1}{\delta-1}}}\,. 
    \]
    This proves the lemma with 
    \[
        C_{\mathrm{out}}=2(4\cdot 5^\delta\tilde{C}_\mu^2)^{\frac{1}{\delta-1}}\,.
    \]
\end{proof}

We shall also make use of the following elementary lemma, which can be checked by calculating all cases explicitly. We omit the proof here.  
\begin{lemma}\label{lem:ramsey}
Let $a_{1},\dots,a_{5}$ be a permutation of $1,2,3,4,5$. Then there exists indices $i_{1}<i_{2}<i_{3}$ such that either $a_{i_{1}}<a_{i_{2}}<a_{i_{3}}$ or $a_{i_{1}}>a_{i_{2}}>a_{i_{3}}$.
\end{lemma}

Now we are ready to prove Proposition \ref{prop:Xis-line-porous}.
\begin{proof}[Proof of Proposition \ref{prop:Xis-line-porous}]

By contradiction, assume that the set $X$ \emph{fails} to be $\nu$-porous on some line segment 
$I \subset \mathbb{R}^2$ of length $r \in [0, \alpha_1]$, where $\nu$ and $\alpha_1$ 
are as given in \eqref{eq:nu-of-X}.

Let $L$ denote the line extending $I$. 
By Lemma~\ref{lem:non-concentration}, we may choose a point 
$q \in \pi(\Lambda) \cap B(0,6)$ such that 
\[
\operatorname{dist}(q, L) \geq \frac{1}{2C_{\text{out}}}.
\]

Since $X$ is not $\nu$-porous on $I$, there exist at least five distinct points 
$\{p_i\}_{1 \le i \le 5}$ in $I$ such that:
\begin{itemize}
    \item The pairwise distances satisfy 
    \[\operatorname{dist}(p_i, p_j) \ge \frac{r}{10}  \qquad \forall\, i \neq j,\]
    \item Each open ball $B(p_i, \nu r)$ contains a point $q_i \in X$.
\end{itemize}
Observe that each $q_i$ lies relatively close to $L$, whereas $q$ is much farther away.

Then the three-point condition \eqref{eq:Euclidean-three-point-condition} together with Lemma \ref{lem:ramsey} guarantees that, among the five points $\{p_i\}$, there exist \emph{three points} — still denoted by $p_1,p_2,p_3$ after discarding the other two points $p_4,p_5$ and the corresponding $q_4,q_5$ — such that:
\begin{itemize}
    \item For each $j=1,2,3$, there exists a point $q_j \in B(p_j,\nu r) \cap X$.
    \item The points $q_1,q_2,q_3,q$ are arranged on $\pi(\Lambda)$ in the cyclic order
    \begin{equation}\label{eq:order-of-p1p2p3q}
        q_{1} \to q_{2} \to q_{3} \to q \to q_{1},
    \end{equation}
    recalling that $\Lambda$ is homeomorphic to $S^1$.
    \item On the line segment $I$, the point $p_2$ lies between $p_1$ and $p_3$.
\end{itemize}
We now replace each $p_i$ by the orthogonal projection of $q_i$ onto $L$; this replacement does not alter the relative positions of $p_1,p_2,p_3$.

The role of the auxiliary point $q$ is to localize the arc condition, thereby controlling the behavior of the arc near $I$. In what follows, we denote by $[q_i,q_j]$ the arc of $\pi(\Lambda)$ connecting $q_i$ and $q_j$ that does \emph{not} pass through $q$. We also use the notation $q \to q_i$ or $q_i \to q$ for the oriented arc joining these two points, consistent with the cyclic order in \eqref{eq:order-of-p1p2p3q}.

Next, we apply Lemma~\ref{lem:non-concentration} to the line $L$ and the ball $B_{\frac{r}{15\tilde{C}_{\operatorname{arc}}}}(q_2)$ of radius $\frac{r}{15\tilde{C}_{\operatorname{arc}}}$ centered at $q_2$. 
Thus we obtain a point $q_4 \in \pi(\Lambda) \cap B_{\frac{r}{15\tilde{C}_{\operatorname{arc}}}}(q_2)$ whose distance to $L$ is at least $\frac{r}{15\tilde{C}_{\operatorname{arc}} C_{\mathrm{out}}}$, while $|q_2 - q_4| \leq \frac{r}{15\tilde{C}_{\operatorname{arc}}}$.

The three-point condition~\eqref{eq:Euclidean-three-point-condition} then forces $q_4$ to lie on the arc $[q_1, q_3]$. 
Indeed, if $q_4 \in [q, q_1]$ (or $q_4 \in [q_3, q]$), then the three-point condition would fail for the pair $(q_2, q_4)$, because the arc joining $q_2$ and $q_4$ would have to pass through either $q_1$ (or $q_3$) or $q$. 
To see this, observe that
\[
\begin{aligned}
|q_2 - q_1| &\geq |p_2 - p_1| - |q_2 - p_2| - |q_1 - p_1| \\
            &\geq \frac{r}{10} - 2\nu r \geq \frac{r}{15},\\[4pt]
|q_2 - q_3| &\geq \frac{r}{15} \qquad (\text{by the same argument as for } |q_2 - q_1|),\\[4pt]
|q_2 - q|   &\geq \operatorname{dist}(q,L) - \nu r \\
            &\geq \frac{1}{2C_{\mathrm{out}}} - \nu r \gg 100r .
\end{aligned}
\]

Hence, after possibly interchanging the roles of $q_1$ and $q_3$, we may assume $q_4 \in [q_1, q_2]$. We prescribe that along $I$, the point $p_1$ lies to the left of $p_2$, and $p_3$ lies to the right of $p_2$. 
 
Define $q_5$ to be the point in the set
\[
    [q_4, q_3] \cap \left\{q \in \mathbb{R}^2 : \operatorname{dist}(q, L) \leq 2\nu r\right\}
\]
whose orthogonal projection onto $L$ (denoted by $p_5$) lies between $p_1$ and $p_2$, and is \emph{farthest} from $p_2$.
We note that this set is nonempty and compact, containing $q_2$ by our choice of the stripe width $\nu$; moreover, $q_5$ may coincide with $q_2$. 
 
Applying the three-point condition \eqref{eq:Euclidean-three-point-condition} to the ordered triple $q_1 \to q_4 \to q_5$ yields 
\[
    \begin{aligned}
        |q_1 - q_5| &\geq \tilde{C}_{\operatorname{arc}}^{-1} |q_4 - q_1| 
                     \geq \frac{r}{20\tilde{C}_{\operatorname{arc}}} \,, \\
        |p_1 - p_5| &\geq |q_1 - q_5| - |q_5 - p_5| - |q_1 - p_1| \\
                    &\geq |q_1 - q_5| - \operatorname{dist}(q_5, L) - \nu r \\
                    &\geq \frac{r}{100\tilde{C}_{\operatorname{arc}}} \,.
    \end{aligned}
\]
Let $p_6$ be the point obtained by moving leftward from $p_5$ along $L$ a distance of $2\nu r$. 
The estimate above shows that along $I$, the points $p_1$, $p_6$, $p_5$, and $p_2$ appear in this order from left to right.
See Figure~\ref{fig:Illustration of our argument} for the configuration of all chosen points.

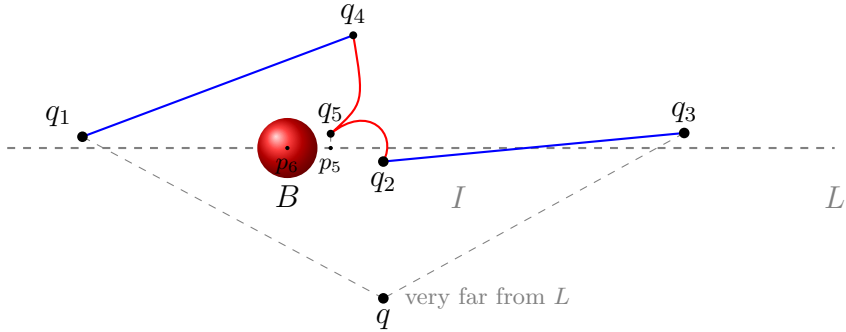
\begin{figure}[ht]
	\centering
	\begin{tikzpicture}
	\coordinate (p1) at (0,0.15);
	\coordinate (p2) at (4,-0.18);
	\coordinate (p3) at (8,0.2);
	
	\coordinate (baseLineLeft) at (-1,0);
	\coordinate (baseLineRight) at (10,0);
	
	\coordinate (q)  at (4,-2);
	\coordinate (p4) at (3.6, 1.5);
	
	
	\draw[dashed, gray, thick] node[below=10pt] at (5,0) {$I$} (baseLineLeft) -- (baseLineRight) node[below=10pt] {$L$};	
	
	\draw[blue, thick] (p1) -- (p4);
	\draw[blue, thick] (p2) -- (p3);
	\draw[gray, dashed] (p3) -- (q);
	\draw[gray, dashed] (q) -- (p1);
	
	
	\coordinate (c1) at (4.2, 0.2);   
	\coordinate (c2) at (3.8, 0.6);   
	\coordinate (c3) at (3.3, 0.19);   
	\coordinate (c5) at (3.75, 0.6);
	\coordinate (c4) at (3.4, 0.75);   
	
	\draw[red, thick, postaction={decorate}] 
		(p2) .. controls (c1) and (c2) .. 
		(c3) .. controls (c5) .. (p4);
	
	\coordinate (p5) at (c3);
	
	
	\coordinate (p5') at (p5 |- 0,0);
	
	\path let \p1 = ($(p4)-(p2)$), \n1 = {veclen(\x1,\y1)/3} in 
	coordinate (p6') at ($(p5')-(\n1,0)$);
	
	
	\shade[ball color=red] (p6') circle (0.4) node[below=10pt] {$B$};
	
	\draw[dashed, gray, thin] (p5) -- (p5');
	
	\fill (p1) circle (2pt);
	\node[above left, font=\normalsize] at (p1) {$q_1$};
	
	\fill (p2) circle (2pt);
	\node[below, font=\normalsize] at (p2) {$q_2$};
	
	\fill (p3) circle (2pt);
	\node[above, font=\normalsize] at (p3) {$q_3$};

	\fill (q) circle (2pt);
	\node[below, font=\normalsize] at (q) {$q$};
	 \node[right=4pt,font=\scriptsize] at (q) {\color{gray} very far from $L$};
	
	\fill (p4) circle (1.5pt) node[above] {$q_4$};
	
	\fill (p5) circle (1.5pt);   
	\node[above, font=\normalsize] at (p5) {$q_5$};
	
	\fill (p5') circle (0.8pt);   
	\node[below, font=\scriptsize] at (p5') {$p_5$};
	
	\fill (p6') circle (0.8pt);   
	\node[below, font=\scriptsize] at (p6') {$p_6$};
	\end{tikzpicture}
\caption{Illustration of our argument in the proof of Proposition~\ref{prop:Xis-line-porous}. The horizontal dashed gray line is $L$, extended from the segment $I$ passing roughly between $q_1$ and $q_3$. Each $q_i$ lies in $X$, while each $p_i$ lies on $I$. The point $q$ belongs to $X$ but is far from $L$. The red ball $B$, together with the points $q_1$, $q_2$, $q_3$, and $q_5$, are all $(\nu r)$-very close to $L$. The point $q_4$ is moderately close to $q_2$ but not moderately-very close to $L$. The point $p_5$ is the projection infimum of the arc $[q_4,q_3]$ and lies on the boundary of $B$, which is a hole in $X$.}
	\label{fig:Illustration of our argument}
\end{figure}

We \emph{claim} that the open 
ball $B$ centered at $p_6$ with radius 
$2\nu r$
does \emph{not} intersect $\pi(\Lambda)$, and hence does not intersect $X$ --- a contradiction. Indeed:\begin{itemize}
    \item For any $x \in B$, its projection onto $L$ lies to the left of $p_5$; 
    by construction of $p_5$, we have $[q_4, q_3] \cap B = \emptyset$.

     \item If $x \in B$ lies on the arc $q \to q_1 \to q_4$, then because $q$ is very far away, 
    $q_4$ must lie on the shorter arc $[x, q_5]$. However,
    \[
        |q_4 - q_5| \geq d(q_4, L) - d(q_5, L) 
        \geq \frac{r}{16 \tilde{C}_{\operatorname{arc}} C_{\mathrm{out}}}
    \]
    and
    \[
        |x - q_5| \leq |x - p_6| + |p_6 - p_5| + |p_5 - q_5| 
        \leq 6\nu r \leq \frac{r}{30 \tilde{C}_{\operatorname{arc}}^2 C_{\mathrm{out}}},
    \]
    contradicting the three-point condition.
\item If $x \in B$ lies on the arc $q_3 \to q$, then
    \[
        |q_5 - q_3| \geq |q_2 - q_3| - 4\nu r \geq \frac{r}{11}, \qquad
        |x - q_5| \leq 6\nu r \leq \frac{r}{30 \tilde{C}_{\operatorname{arc}}^2 C_{\mathrm{out}}},
    \]
    contradicting the three-point condition applied to $q_5 \to q_3 \to x$.
\end{itemize}

Thus, there exists a ball of radius $2\nu r$ centered at $p_6$ on $I$ that does not intersect $X$, yielding a contradiction. Therefore, $X$ must be $\nu$-porous on any line segment $I \subset \mathbb{R}^2$ of length $r \in [0, \alpha_1]$, which completes the proof.
\end{proof}

\subsection{Proof of Theorem \ref{thm:Spectral}} Following the argument in \cite[Theorem 3]{dyatlov2016spectral}, it suffices to prove the hyperbolic FUP stated in Proposition \ref{prop:hyperfup}.  
By applying stereographic projection, we reduce further to proving Proposition \ref{prop:hyperFUPEuclidean}—a Euclidean version of the hyperbolic FUP with the exponent $\beta$ given in \eqref{eq:beta-hyper}.   
The passage from the standard FUP \eqref{eq:fup} and its symmetric form \eqref{eq:fup-sym} to the hyperbolic FUP \eqref{eq:hyperfup} and its Euclidean counterpart \eqref{eq:hyperbolicFUP-euclidean} is now routine; see \cite[Section 4]{bourgain2018spectral} for the one-dimensional case and \cite[Section 4]{kim2025semiclassical} for higher dimensions.  
The essential idea is to extend the standard FUP to Fourier integral operators.  
For the reader's convenience, we list below the key lemmas in this reduction and refer to \cite{bourgain2018spectral,kim2025semiclassical} for complete details.

Recall the symmetric FUP \eqref{eq:fup-sym}.  
Let $X_+,X_-\subset B_{\mathbb{R}^2}(0,R)$ be two sets that are $\nu$-porous on lines at all scales from $0$ to $1$, where $R>0$ is a large constant.  Then the quantitative FUP gives 
\[
    \|\mathbf 1_{X_-(h)}\mathcal{F}_h\mathbf 1_{X_+(h)}\|_{L^2\to L^2}\leq Ch^\beta,
\]
where $\mathcal{F}_h$ is the semiclassical Fourier transform
\[
    \mathcal{F}_hf(\xi)=(2\pi h)^{-1}\int_{\mathbb{R}^2} e^{- i\,x\cdot\xi/h}f(x)\,dx,
\]
the exponent $\beta=\beta(\nu)$ is given explicitly, and constant $C$ depends only on $R$ and $\nu$. 

The first lemma is an FUP for operators with variable amplitude; it is essentially the same as \cite[Proposition 4.2]{bourgain2018spectral} and \cite[Lemma 4.13]{kim2025semiclassical} (and the proof follows similarly). Consider operators \(A = A_a(h) : L^2(\mathbb{R}^2) \to L^2(\mathbb{R}^2)\) of the form
\[
A_a(h)f(x) =(2\pi h)^{-1} \int_{\mathbb{R}^2} e^{- i x \cdot \xi / h} \, a(x, \xi) \, f(\xi) \, d\xi,
\]
where \(a(x, \xi) \in C_c^\infty(\mathbb{R}^2)\) satisfies
\[
\sup | \partial_x^k a | \leq C_{k,a}, \qquad \operatorname{diam}\bigl(\supp a \bigr) \le C_a
\]
for each multi-index \(k\) and for some constants \(C_{k,a}, C_a\).
\bigskip 
\begin{lemma}
Let $\beta=\beta(\nu)$ and $C_1>0$.
    For each $\rho \in (1/2,1]$, there exists a constant $C$ depending only on $C_{k,a},C_a,R,\nu,\rho$ and $C_1$ 
such that for all $h$
\[
\|\mathbf 1_{X_-(C_1h^\rho)} A(h) \mathbf 1_{X_+(C_1h^\rho)} \|_{L^2(\mathbb R^2) \to L^2(\mathbb R^2)} \leq Ch^{\beta - 2(1 - \rho)}\,.
\]

\end{lemma} 
\begin{remark} 
The loss of exponent $2(1-\rho)$ is due to the fact that we need to use $h^{\rho-1}$ number of sets of scale $h$ to cover $X_{\pm}(C_1h^\rho)$.
\end{remark}

Next we need to consider the Fourier integral operator with general phase. 
Let \(B=B_{\Phi, \chi}(h):L^{2}(\mathbb{R}^{2})\to L^{2}(\mathbb{R}^{2})\) be of the form
\[
B_{\Phi, \chi}(h)f(x)=(2\pi h)^{-1}\int_{\mathbb{R}^{2}}e^{i\Phi(x,x')/h}\chi(x,x')f(x')\,dx',
\]
where we assume
\[
     \Phi(x,x')\in C^{\infty}(B(0,R);\mathbb{R}),\quad \chi\in C_{c}^{\infty}(B(0,R)),\quad\text{det }(\partial_{xx'}^{2}\Phi)\neq 0\text{ on } B(0,R)\,.
\]
Note that we use notation
\[\partial^2_{xx'}\Phi=
\left(
\begin{array}{cc}
 \partial^2_{x_1x'_1} \Phi &    \partial^2_{x_1x'_2} \Phi \\
\partial^2_{x_2x'_1} \Phi &   \partial^2_{x_2x'_2} \Phi  \\ 
\end{array}
\right)\,.\]

Then we have the following FUP on the operator $B$, which is essentially the same as \cite[Proposition 4.3]{bourgain2018spectral} and \cite[Proposition 4.14]{kim2025semiclassical}. The proof 
is exactly the same as \cite[Proposition 4.14]{kim2025semiclassical}.

\begin{proposition}
\label{prop:FUP for general FIO,Kim's version}
    Suppose there exists a constant $C_2\geq 1$ and a neighborhood $U$ of $\operatorname{supp} b$
    such that
    \begin{equation}\label{eq:condition on FUP for FIO,kim's version}
        \max(\|\partial_{xx'}^2\Phi\|,\|(\partial_{xx'}^2\Phi)^{-1}\|)\leq C_2  \quad \text{in} \  U
    \end{equation}
    where $\|\cdot\|$ is the operator norm of an $2\times 2$ matrix under $\ell^2$ norm in $\mathbb R^2$.
    Then
    for $\rho\in(3/4,1)$ and 
    \[
        \beta=\beta\left(\frac{\nu}{2C_2^2}\right)
    \]
    there exists a constant $C>0$ independent of $h$, 
    such that for all $h$
\[
\|\mathbf 1_{X_{-}(C_1h^{\rho})}B(h)\mathbf 1_{X_{+}(C_1h^{\rho})}\|_{L^{2}(\mathbb{R}^{2})\to L^{2}(\mathds{R}^{2})}\leq Ch^{\beta/2-2(1-\rho)}.
\]
\end{proposition}
\begin{remark}
According to the comment after 
\cite[Proposition 4.14]{kim2025semiclassical}, 
the exponent $\beta$ depends only on $\nu$ and is independent of $\Phi$
in the 1-dimensional version of Proposition \ref{prop:FUP for general FIO,Kim's version}, as proved in \cite[Proposition 4.3]{bourgain2018spectral}. In higher dimensional case, since the transform under $\partial_x'\Phi$ will change the porous constant $\nu$, 
the exponent $\beta$ depends on the phase function $\Phi$.

However, we can not directly apply Proposition \ref{prop:FUP for general FIO,Kim's version} to  prove \eqref{eq:hyperbolicFUP-euclidean}, since for the specific 
$\Phi$ in \eqref{eq:B_b}, the norm of $\partial^2_{zw}\Phi$ is very large near the diagonal $z=w$. Fortunately, we can use the scaling trick used in the 
proof of \cite[Proposition 4.3]{bourgain2018spectral} to obtain a relatively uniform exponent $\beta$, with weaker restriction on phase.
\end{remark}
\begin{proposition}
\label{prop:FUP_FIO_our_version} 
Suppose there exists a constant $C_3\geq 1$ and a neighborhood $U$ of $\operatorname{supp} b$
    such that 
    \begin{equation}
    \label{eq:PhaseCond2}
        \frac{s_{\max}(\partial^2_{xx'}\Phi)}{s_{\min}(\partial^2_{xx'}\Phi)}\leq C_3 \quad \text{in} \ U
    \end{equation}
    where $s_{\max}$ and $s_{\min}$ are the largest and the smallest 
    singular values. Then for  $\rho\in(3/4,1)$ and 
 \[ \beta=\beta\left(\frac{\nu}{4C_3}\right)
    \]
    there exists a constant $C>0$ independent of $h$, such that for all $h$
\[
\|\mathbf 1_{X_{-}(C_1h^{\rho})}B(h)\mathbf 1_{X_{+}(C_1h^{\rho})}\|_{L^{2}(\mathbb{R}^{2})\to L^{2}(\mathds{R}^{2})}\leq Ch^{\beta/2-2(1-\rho)}.
\]
\end{proposition}
\begin{proof}
 We first cover the support by  finite small balls.  For each $(x_0,x'_0)\in \operatorname{supp} b$, we can find a small neighborhood ball $B_0=B(x_0,x'_0)\subset \mathbb{R}^4$ such that, for all $(x,x')\in B_0$, we have 
   \[
        M/\sqrt 2<{s_{\max}(\partial^2_{xx'}\Phi(x,x'))}<\sqrt 2M,\quad 
         m/\sqrt 2<{s_{\min}(\partial^2_{xx'}\Phi(x,x'))}<\sqrt 2m
    \]
where $M:=s_{\max}(\partial^2_{xx'}\Phi(x_0,x_0'))$ and $m:=s_{\min}(\partial^2_{xx'}\Phi(x_0,x_0'))$.    
Let 
    $\lambda:=\sqrt{Mm}\ge c_\Phi>0$. Now, we consider the new phase function 
    $\tilde{\Phi}$ defined as $\tilde{\Phi}=\lambda^{-1} \Phi$. We have $B(h)=\lambda^{-1}\tilde{B}(\lambda^{-1}h)$
where the new operator $\tilde{B}(h)$ defined as 
    \[
        \tilde{B}(h)f(x)=(2\pi h)^{-1}\int_{\mathbb{R}^{2}}e^{i\tilde{\Phi}(x,x')/h}\chi(x,x')f(x')\,dx',
    \]
    for any $b$ supported on $B_0$.
 Since $\tilde{\Phi}$ satisfies the condition \eqref{eq:condition on FUP for FIO,kim's version} with $C_2=\sqrt{2C_3}$ and 
by applying partition of unity 
on general $b$,  we conclude the 
 desired estimate for $\tilde{B}(h)$ and then for $B(h)$ for exponent $\beta=\beta(\frac{\nu}{4C_3})$ and a very large constant $C$ independent of $h$, but dependent of $\Phi$. 
\end{proof}
\begin{remark}
Note here that $\beta$ also depends on $\Phi$.
\end{remark}

We are now ready to prove Theorem \ref{thm:Spectral}.
\begin{proof}[Proof of Proposition \ref{prop:hyperFUPEuclidean} and Theorem \ref{thm:Spectral}]
It suffice to prove \eqref{eq:hyperbolicFUP-euclidean} with exponent $\beta$ given by \eqref{eq:beta-hyper}. We will use Proposition \ref{prop:FUP_FIO_our_version} for specific phase. By the definition of  the phase function $\Phi$ in \eqref{eq:B_b}, we can directly compute
    \[
        \partial^2_{xx'}\Phi=\frac{1}{R^2}(-I+2{v}{v}^T)
    \]
    where 
  \[R=|x-x'|>0, \quad v=(x-x')/R\in \mathbb{S}^1\subset \mathbb{R}^2.\] Note that this second-order derivative matrix blows up in the diagonal. However it still satisfies the phase condition \eqref{eq:PhaseCond2} with $C_3=1$. 

Now we apply Proposition \ref{prop:FUP_FIO_our_version} for the operator $B=B_\chi(h)$ with the specific phase $\Phi$ in \eqref{eq:B_b} which satisfies the condition \eqref{eq:PhaseCond2} with \[ C_3=1,\] and for the sets
\[X_-=X_+=10C_{\mathrm{out}}\tilde{C}_{\operatorname{arc}}X\]
which is $\nu$-porous on lines from scales $0$ to $1$ by Proposition \ref{prop:Xis-line-porous}, where porosity parameter $\nu$ is given by \eqref{eq:nu-of-X}. We conclude the FUP estimate \eqref{eq:hyperbolicFUP-euclidean} with $\beta=\frac12\beta(\nu/4)$.
    
    According to \eqref{eq:definition-of-Cout} and \eqref{eq:tildeCmu and tildeCarc}, we have estimate
 \[\nu>\frac{1}{2\cdot10^7(10^{11}C_\mu^2)^{1/(\delta-1)}C_{\operatorname{arc}}^2}\,.\] 
So we can actually take
    \begin{equation*}\label{eq:explicit beta in the last proof}
        \beta=\frac12\,\beta\left(\frac{1}{10^8(10^{11}C_\mu^2)^{1/(\delta-1)}C_{\operatorname{arc}}^2}\right)\,,
    \end{equation*}
    where $C_\mu$ and $C_{\operatorname{arc}}$ is defined as in \eqref{eq:delta_regularity_of_Lambda} and \eqref{eq:three_point_condition}.
    
    This completes the proof of the hyperbolic FUP---Proposition  \ref{prop:hyperFUPEuclidean} with exponent $\beta$ in \eqref{eq:beta-hyper} and therefore Theorem \ref{thm:Spectral} with the same
explicit spectral gap $\beta$.
\end{proof}

\subsection{FUP for circles and the resonances for Fuchsian groups}\label{subsec:fuch-Appendix-B}
We recall from Example~\ref{ex:circle} that FUP on a circle cannot follows from Theorem \ref{thm:qfup}.  Consequently, the argument presented in the preceding subsections fails for Fuchsian groups. In this subsection, we demonstrate the existence of an essential spectral gap for Fuchsian hyperbolic 3-manifolds by explicitly determining all poles of the resolvent via separation of variables. This analysis applies when the limit set is a genuine circle, in which case the quotient manifold $M = \Gamma \backslash \mathbb{H}^{3}$ is isometric to a warped product space. 

\medskip

\subsubsection*{Warped product structure}

Let \(\Gamma\) be a convex cocompact Fuchsian group of the first kind, so that its limit set \(\Lambda(\Gamma)\) is the entire circle \(\mathbb{S}^{1} \subset \partial \mathbb{H}^{3}\). Let \(CH(\Lambda)\) denote the convex hull of \(\Lambda(\Gamma)\) in \(\mathbb{H}^{3}\). As shown in \cite[Section 6.1--6.2]{ratcliffe2006foundations}, \(CH(\Lambda)\) is a totally geodesic hyperbolic plane isometric to \(\mathbb{H}^{2}\). Since \(\Gamma\) preserves \(CH(\Lambda)\), the quotient
\[
\Sigma := \Gamma \backslash CH(\Lambda)
\]
is a compact hyperbolic surface (the absence of parabolic elements in \(\Gamma\) guarantees compactness). The normal exponential map of \(\Sigma\) in \(M\) gives a global diffeomorphism
\[
\exp : N\Sigma \simeq \Sigma \times \mathbb{R} \longrightarrow M,
\]
where \(N\Sigma\) is the normal bundle of \(\Sigma\). Because \(\Sigma\) is totally geodesic, the induced metric on \(M\) takes the explicit warped product form
\begin{equation}\label{eq:warped-product-metric}
M \cong \bigl( \mathbb{R}_t \times \Sigma_x,\; dt^{2} + \cosh^{2}t \; dS(x) \bigr),
\end{equation}
where \(dS\) is the Riemannian metric on \(\Sigma\). See also \cite[Appendix A]{MR3077910} for details.

\medskip

\subsubsection*{Resolvent and separation of variables}

The warped product structure \eqref{eq:warped-product-metric} allows us to compute the resonances of the Laplace--Beltrami operator \(-\Delta\) on \(M\) explicitly by separation of variables. Consider the resolvent
\[
\mathcal{R}(s) = \bigl(-\Delta - s(2-s)\bigr)^{-1},
\]
initially defined for \(\operatorname{Re} s > 1\) and then meromorphically continued to \(\mathbb{C}\). In the warped product coordinates \((t,x)\) the Laplacian reads
\begin{equation}\label{eq:Laplacian-warped}
-\Delta = - \partial_t^{2} - 2\tanh t \, \partial_t + \frac{1}{\cosh^{2}t} (-\Delta_{\Sigma}),
\end{equation}
where \(\Delta_{\Sigma}\) is the Laplace--Beltrami operator on the compact hyperbolic surface \(\Sigma\).

Let \(\{\phi_k\}_{k=0}^{\infty}\) be an orthonormal eigenbasis of \(L^{2}(\Sigma)\) satisfying
\[
-\Delta_{\Sigma} \phi_k = \mu_k \phi_k, \qquad 
0 = \mu_0 < \mu_1 \le \mu_2 \le \dots \to +\infty.
\]
Expanding any function \(u(t,x) \in C^{\infty}(M)\) as \(u(t,x)=\sum_k u_k(t)\phi_k(x)\) and by product representation \eqref{eq:Laplacian-warped}, the equation \((-\Delta - s(2-s))u = v\) reduces, for each \(k\), to the ordinary differential equation
\begin{equation}\label{eq:radial-ode}
-u_k''(t) - 2\tanh t \, u_k'(t) + \Bigl( -s(2-s) + \frac{\mu_k}{\cosh^{2}t} \Bigr) u_k(t) = v_k(t).
\end{equation}
The substitution \(u_k(t) = (\cosh t)^{-1} f_k(t)\) transforms \eqref{eq:radial-ode} into
\begin{equation}\label{eq:poschl-teller}
-f_k''(t) + \frac{\mu_k}{\cosh^{2}t}\, f_k(t) + (s-1)^{2} f_k(t) = v_k(t).
\end{equation}

Equation \eqref{eq:poschl-teller} is a one-dimensional Schr\"{o}dinger equation with the P\"{o}schl--Teller potential \(V_k(t)=\mu_k/\cosh^{2}t\) and energy parameter \(E = -(s-1)^{2}\) (see \cite[Section 5.1]{borthwick2016spectral}). Denote by
\[
\mathcal{R}_k(\lambda) := \bigl(-\partial_t^{2} + V_k - \lambda^{2}\bigr)^{-1}, \qquad \lambda = i(s-1),
\]
the meromorphically extended resolvent of the Schr\"odinger operator \(-\partial_t^{2}+V_k\) on \(\mathbb{R}\). The scattering poles of \(\mathcal{R}_k\) are classical: for \(k \ge 1\) they occur at
\[
\lambda = \pm \sqrt{\mu_k - \tfrac14} - i\bigl(n+\tfrac12\bigr), \qquad n \in \mathbb{N},
\]
while for \(k=0\) the only pole is at \(\lambda = 0\) (i.e., \(s=1\)). Because the full resolvent \(\mathcal{R}(s)\) on \(M\) decomposes as the direct sum 
\[\mathcal{R}(s)=(\cosh t)^{-1}\bigoplus_k  \mathcal{R}_k(i(s-1))\,,\] its resonances are obtained from these one‑dimensional data. For \(k \ge 1\) the poles \(s_{k,n}^{\pm}\) are given explicitly by
\begin{equation}\label{eq:explicit-poles}
s_{k,n}^{\pm} = \frac12 - n \pm \sqrt{\frac14 - \mu_k}, \qquad n \in \mathbb{N}.
\end{equation}
Together with the possible simple pole at \(s=1\) (coming from \(k=0\)), these constitute all resonances of \(-\Delta\) on \(M\).

\medskip

\subsubsection*{Spectral gap}
Since the eigenvalues \(\mu_k \to \infty\) as \(k \to \infty\), formula \eqref{eq:explicit-poles} shows that all resonances except $s=1$ satisfy \(\operatorname{Re} s \le \frac12\). Hence the operator \(-\Delta\) possesses an essential spectral gap:  there are finite resonances with \(\operatorname{Re} s > \frac12\). Translating the spectral parameter to the convention \(-\Delta - 1 - \lambda^{2}\) used in Theorem~\ref{thm:Spectral}, we obtain the following result.

\begin{proposition}\label{prop:Fuchsian-gap}
Let \(M = \Gamma \backslash \mathbb{H}^{3}\) be a convex cocompact hyperbolic \(3\)-manifold whose limit set \(\Lambda(\Gamma)\subset \mathbb{S}^{2}\) is a genuine circle. Then the meromorphic resolvent
\[
R(\lambda) = \bigl(-\Delta - 1 - \lambda^{2}\bigr)^{-1}
\]
has at most one pole $\lambda=0$ in 
 the half‑plane \(\operatorname{Im}\lambda > -\frac12\). That is, $M$  has an essential spectral gap of size $\beta=\frac12$. 
\end{proposition}

\begin{remark}
This proposition together with Theorem \ref{thm:Spectral}  resolves the spectral gap problem for \emph{all} quasi-Fuchsian groups. In analogy with Theorem~\ref{thm:Spectral} we expect the following cutoff resolvent estimate: 
\[
\bigl\| \chi R(\lambda) \chi \bigr\|_{L^{2}(M) \to L^{2}(M)}
\le \widetilde{C}_{\chi,\epsilon} \,
|\lambda|^{-1-2\min(0,\operatorname{Im}\lambda)+\epsilon},
\qquad
\operatorname{Im}\lambda \in \bigl[-\tfrac12+\epsilon,\;1\bigr],\;
|\operatorname{Re}\lambda| \ge C_{\epsilon}.
\]
The factor \(|\lambda|^{-1-2\min(0,\operatorname{Im}\lambda)+\epsilon}\) reflects the polynomial decay of the resolvent in the specified strip, modulo an arbitrarily small loss of \(\epsilon\). 
While a purely spectral argument would yield a resolvent bound (perhaps with a non‑sharp exponent), we present below a direct hyperbolic FUP for circles that yields the sharp gap.
\end{remark}

\medskip
\subsubsection*{Hyperbolic FUP for circles} Assume the limit set is the equator
\[
\Lambda = \bigl\{ (x,y,z)\in \mathbb{S}^{2}\subset\mathbb{R}^{3} : x^{2}+y^{2}=1,\;z=0 \bigr\}.
\]
Introduce spherical coordinates \(w = w(\theta, \phi)\) with \(\theta \in (-\pi/2,\pi/2)\) and \(\phi \in (0,2\pi)\) by
\[
 x = \cos\theta\cos\phi,\qquad y = \cos\theta\sin\phi,\qquad  z = \sin\theta.
\]
For \(w' = w'(\theta',\phi') \in \Lambda_{(Ch^{\rho})} \subset \mathbb{S}^{2}\) we write
\begin{align*}
\mathcal{B}_{\chi}(h)\bigl(\mathbf{1}_{\Lambda(Ch^{\rho})}f\bigr)(w)
&= (2\pi h)^{-1}\int_{\mathbb{S}^{2}}
|w-w'|^{2i/h}\,\chi(w,w')\,\mathbf{1}_{\Lambda(Ch^{\rho})}(w')\,f(w')\,dw' \\
&= (2\pi h)^{-1}\int_{-2Ch^{\rho}}^{2Ch^{\rho}}\int_{0}^{2\pi}
|w-w'|^{2i/h}\,\chi(w,w')\,\mathbf{1}(|\sin\theta'|\le Ch^{\rho})\\
&\hskip 5cm \times f(\theta',\phi')\,\cos\theta'\,d\phi'\,d\theta'.
\end{align*}
Expanding \(\chi\) around \(\theta=\theta'=0\) gives
\begin{align*}
\mathcal{B}_{\chi}(h)\bigl(\mathbf{1}_{\Lambda(Ch^{\rho})}f\bigr)(w)
&= (2\pi h)^{-1}\int_{-2Ch^{\rho}}^{2Ch^{\rho}}\int_{0}^{2\pi}
|w-w'|^{2i/h}\,\chi_{0}(\phi,\phi')\,\mathbf{1}(|\sin\theta'|\le Ch^{\rho})\\
&\hskip 5cm \times f(\theta',\phi')\,d\phi'\,d\theta' \;+\; R(h)f \\
&= h^{-1/2}\int_{-2Ch^{\rho}}^{2Ch^{\rho}}
T(h,\theta,\theta')\bigl(f(\theta',\cdot)\bigr)(\phi)\,
\mathbf{1}(|\sin\theta'|\le Ch^{\rho})\,d\theta' \;+\; R(h)f ,
\end{align*}
where \(\chi_{0}(\phi,\phi') = \chi(\theta=0,\theta'=0,\phi,\phi')\),
 \(R(h) = \mathcal{O}(h^{\rho})_{L^{2}(\mathbb{S}^{2})\to L^{2}(\mathbb{S}^{2})}\),
and \(T(h,\theta,\theta')\) is a one‑dimensional Fourier integral operator defined for \(f\in C^{\infty}(\mathbb{S}^{1})\) by
\[
T(h,\theta,\theta')f(\phi)
= (2\pi)^{-1}h^{-1/2}
\int_{0}^{2\pi} \exp\bigl(i\Phi_{\theta,\theta'}(\phi,\phi')/h\bigr)\,
\chi_{0}(\phi,\phi')\,f(\phi')\,d\phi',
\]
with phase
\begin{align*}
\Phi_{\theta,\theta'}(\phi,\phi') =
\ln\!\bigl[&(\sin\theta-\sin\theta')^{2}
+ (\cos\theta\cos\phi-\cos\theta'\cos\phi')^{2} \\
&+ (\cos\theta\sin\phi-\cos\theta'\sin\phi')^{2}\bigr].
\end{align*}
The operator \(T\) is \(L^{2}(\mathbb{S}^{1}_{\phi'})\to L^{2}(\mathbb{S}^{1}_{\phi})\) bounded,  uniformly for sufficiently small \(\theta,\theta'\). Indeed, when \(\theta=\theta'=0\) it reduces to the model operator in \cite[(1.6)]{dyatlov2016spectral} with \(n=2\); for small \(\theta,\theta'\) the phase remains non‑degenerate on the support of \(\chi_{0}\), and uniform \(L^{2}\)‑boundedness follows from a standard \(TT^{*}\) argument.

Set \(g(w) = \mathbf{1}_{\Lambda(Ch^{\rho})}\mathcal{B}_{\chi}(h)\mathbf{1}_{\Lambda(Ch^{\rho})}f(w)-R(h)f(w)\). By Cauchy–Schwarz inequality,
\begin{align*}
\|g\|_{L^{2}(\mathbb{S}^{2})}^{2}
&\le \int_{-2Ch^{\rho}}^{2Ch^{\rho}}\int_{0}^{2\pi} |g(\theta,\phi)|^{2}\,d\phi\,d\theta \\
&\le \int_{-2Ch^{\rho}}^{2Ch^{\rho}} h^{-1}
\int_{0}^{2\pi}
\Bigl(\int_{-2Ch^{\rho}}^{2Ch^{\rho}}
\bigl|T(h,\theta,\theta')f(\theta',\cdot)(\phi)\bigr|\,d\theta'\Bigr)^{2}
d\phi\,d\theta \\
&\le \int_{-2Ch^{\rho}}^{2Ch^{\rho}} 2Ch^{-1+\rho}
\int_{0}^{2\pi}\int_{-2Ch^{\rho}}^{2Ch^{\rho}}
\bigl|T(h,\theta,\theta')f(\theta',\cdot)(\phi)\bigr|^{2}
d\theta'\,d\phi\,d\theta \\
&= \int_{-2Ch^{\rho}}^{2Ch^{\rho}} 2Ch^{-1+\rho}
\int_{-2Ch^{\rho}}^{2Ch^{\rho}}
\bigl\|T(h,\theta,\theta')f(\theta',\cdot)\bigr\|_{L^{2}(\mathbb{S}^{1}_{\phi})}^{2}
d\theta'\,d\theta \\
&\lesssim \int_{-2Ch^{\rho}}^{2Ch^{\rho}} h^{-1+\rho}
\int_{-2Ch^{\rho}}^{2Ch^{\rho}}
\int_{0}^{2\pi}|f(\theta',\phi')|^{2}\,d\phi'\,d\theta'\,d\theta \\
&\lesssim h^{-1+2\rho}\,\|f\|_{L^{2}(\mathbb{S}^{2})}^{2}.
\end{align*}
Hence estimate \eqref{eq:hyperfup} holds with exponent \(\beta = 1/2\) and any \(\epsilon>0\). Summarising,
\begin{proposition}
If \(\Lambda = \mathbb{S}^{1}\), then
\[
\bigl\|\mathbf{1}_{\Lambda(C_{1}h^{\rho})}
\mathcal{B}_{\chi}(h)\mathbf{1}_{\Lambda(C_{1}h^{\rho})}
\bigr\|_{L^{2}(\mathbb{S}^{2})\to L^{2}(\mathbb{S}^{2})}
\le C h^{1/2-\varepsilon},\qquad h\in(0,1/100).
\]
Consequently, the resolvent possesses a spectral gap of size \(1/2\) and satisfies
\[
\bigl\| \chi R(\lambda) \chi \bigr\|_{L^{2}(M)\to L^{2}(M)}
\le \widetilde{C}_{\chi,\epsilon}\,
|\lambda|^{-1-2\min(0,\operatorname{Im}\lambda)+\epsilon},
\]
for all \(\lambda\) with \(\operatorname{Im}\lambda \in [-\tfrac12+\epsilon,\,1]\) and \(|\operatorname{Re}\lambda| \ge C_{\epsilon}\).
\end{proposition}

\begin{remark}
The rough Fourier decay of the circle measure alone does not yield the sharp exponent for the FUP. Nevertheless, the same separation‑of‑variables technique shows that for any \(C_{1}>0\) there exists \(C>0\) such that, when \(X = \mathbb{S}^{1}\subset \mathbb{R}^{2}\),
\[
\bigl\|\mathbf{1}_{X(C_{1}h)}\mathcal{F}_{h}\mathbf{1}_{X(C_{1}h)}\bigr\|
_{L^{2}(\mathbb{R}^{2})\to L^{2}(\mathbb{R}^{2})}
\le C h^{-1/2}.
\]
\end{remark}

\appendix \section{Some useful lemmas} \label{sec:appA}
For the reader's convenience, we present several technical lemmas from \cite{MR4927737}, concerning the relations, transformations, restrictions, and measures of porous sets.
\begin{lemma}[\cite{MR4927737} Lemma A.3]
\label{lem:compare_porosity}
Let $\mathbf X \subset [-1,1]^d$ be $\nu$-porous on balls from scale $h$ to $1$. Then $\mathbf X$ is box porous at scale $L = \lceil \nu^{-1}\sqrt{d}\rceil$ with depth $n$ for all $n\geq 0$ with $L^{-n} \geq h$. 
\end{lemma}

\begin{lemma}[\cite{MR4927737} Lemma A.5] \label{lem:porous_transf}
    Let \( \mathbf{X} \subset \mathbb{R}^d \) be \(\nu\)-porous on lines from scales \(\alpha_0\) to \(\alpha_1\).
    \begin{enumerate}[label=(\alph*), leftmargin=*, nosep]
        \item Let \(\alpha_0 < r < \alpha_1\) and let \(\nu' < \nu\). Then \(\mathbf{X} + B_r\) is \(\nu'\)-porous on lines from scales \(r/(\nu - \nu')\) to \(\alpha_1\).
        \item For any \(s > 0\), the dilate \(s \cdot \mathbf{X}\) is \(\nu\)-porous on lines from scales \(s\alpha_0\) to \(s\alpha_1\).
        \item Let \(\ell \subset \mathbb{R}^d\) be a line. Let \(\mathbf{X}|_\ell = \mathbf{X} \cap \ell\), and view \(\mathbf{X}|_\ell\) as a subset of \(\mathbb{R}\). Then \(\mathbf{X}|_\ell\) is \(\nu\)-porous from scales \(\alpha_0\) to \(\alpha_1\).
    \end{enumerate}
\end{lemma}

\begin{lemma}[\cite{MR4927737} Lemma A.6] \label{lem:measure-boxporous}
Let $\mathbf X\subset [-1,1]^d$ be box porous at scale $L$ with depth $n$ for all $0\le n<N$. Then the Lebesgue measure
 \[|\mathbf X|\le 2^d (1-L^{-d})^N\,.\]
 \end{lemma}

\begin{lemma}[\cite{MR4927737} Lemma A.7] \label{lem:A7}
    Let \(\mathbf{X} \subset \mathbb{R}^{d}\) be \(\nu\)-porous on balls from scales \(\alpha_0\) to \(\alpha_1\). Then 
    for any ball \(B\) of radius \(\alpha_0 < R < \alpha_1\),
    \[
    |\mathbf{X} \cap B| \leq C_d R^{d} \left( \frac{\alpha_0}{R} \right)^{\gamma}.
    \]
    where $\gamma$ can be taken as
    \[
        \gamma =\frac{\nu^{d}}{C_d(1+|\log \nu|)}
    \]
\end{lemma}

\begin{proof}
    We use Lemma~\ref{lem:compare_porosity} to reduce to proving the statement for box porous sets. Let \(Q\) be a \(2R\)-cube containing \(B\). Let \(\mathbf{X}' = R^{-1} \cdot (Q \cap \mathbf{X}) \subset [-1, 1]^{d}\) be a translated and rescaled copy of \(Q \cap \mathbf{X}\). Then \(\mathbf{X}'\) is \(\nu\)-porous on balls from scales \(\frac{\alpha_0}{R}\) to \(1\), so it is also box porous at scale \(L = \lceil \nu^{-1} \sqrt{d} \rceil\) with depth \(n\) for all \(n \geq 0\) with \(L^{-n} \geq \frac{\alpha_0}{R}\). Let \(N > 0\) be the smallest integer such that \(L^{-N} < \frac{\alpha_0}{R}\). By Lemma~\ref{lem:measure-boxporous}    
    \[
    |\mathbf{X} \cap B| \leq R^{d} |\mathbf{X}'| \leq 2^{d} R^{d} \left( 1 - L^{-d} \right)^{N}.
    \]
    Let \(\gamma = \gamma(\nu, d) > 0\) be such that \(L^{-\gamma} = 1 - L^{-d}\). Then
    \[
    |\mathbf{X} \cap B| \leq 2^{d} R^{d} L^{-N \gamma} \leq 2^{d} R^{d} \left( \frac{\alpha_0}{R} \right)^{\gamma}.
    \]
    Then 
    \[
        \gamma=\frac{
        \log\left(\frac{1}{1-L^{-d}}\right)}{\log L}\geq \frac{L^{-d}}{C_d\log L}\geq \frac{\nu^{d}}{C_d(1+\log \nu^{-1})}
    \]
\end{proof}

\begin{lemma}[\cite{MR4927737} Corollary A.8]\label{lem:measurelineintersection}
    Let \(\mathbf{Y} \subset \mathbb{R}^{d}\) be \(\nu\)-porous on lines from scales \(\alpha_0\) to \(\alpha_1\). Let \(\tau\) be a line segment of length \(\alpha_0 < R < \alpha_1\). Then there is some \(C, \gamma > 0\) depending only on \(\nu\) such that
    \[
    |\tau \cap \mathbf{Y}| \leq C R \left( \frac{\alpha_0}{R} \right)^{\gamma}.
    \]
    where $\gamma$ can be taken as
    \[
        \gamma =\frac{\nu}{C(1+|\log \nu|)}
    \]
\end{lemma}

\begin{proof}
    Let \(\tau\) lie on the line \(\ell\). By Lemma~\ref{lem:porous_transf},  \(\mathbf{X}|_{\ell}\) is \(\nu\)-porous. By Lemma~\ref{lem:A7} in \(d = 1\) we obtain the result.
\end{proof}

\vskip2mm
\subsection*{Acknowledgements.} 
Long Jin and Hong Zhang are supported by National Key R\&D Program of China No. 2022YFA1007400. Long Jin is also supported by National Natural Science Foundation of China  No. 12525106 and New Cornerstone Investigator Program 100001127.
An Zhang is partially supported by National Key R\&D Program of China No. 2024YFA1015300, Beijing Natural Science Foundation No. 1242009, National Natural Science Foundation of China No. 11801536, the China Scholarship Council No. 202506020208 and the Fundamental Research Funds for the Central Universities. The authors would like to thank Kiril Datchev, Semyon Dyatlov and Ruixiang Zhang for inspiring discussions.

\subsection*{Statements.} The authors have no relevant financial or non-financial interests to disclose. Data sharing is not applicable to this article as no datasets were used.

\bibliographystyle{alpha}
\bibliography{ref}
\end{document}